\documentclass[12pt,preprint]{amsart}

\usepackage[T2A]{fontenc}
 \usepackage[utf8]{inputenc}
\usepackage{geometry}
\usepackage{natbib} 
\usepackage{amsmath,amsthm,verbatim,amssymb,amscd, graphicx}
\usepackage{mathrsfs}%
\usepackage[title]{appendix}%
\usepackage{xcolor}%
\usepackage{textcomp}%
\usepackage{manyfoot}%
\usepackage{booktabs}%
\usepackage{algorithm}%
\usepackage{algorithmicx}%
\usepackage{algpseudocode}%
\usepackage{listings}%

\usepackage{bbm}
\usepackage{verbatim}
\usepackage{xcolor}
\usepackage{hyperref}
\usepackage{cleveref}
\usepackage{bm} 
\usepackage{comment}
\usepackage{graphicx}
\usepackage{subcaption}

\newtheorem{theorem}{Theorem}[section]
\newtheorem{corollary}{Corollary}[section]
\newtheorem{remark}{Remark}
\newtheorem{lemma}{Lemma}[section]

\newtheorem{prop}{Proposition}[section]

\newcommand{\eps}{\varepsilon}
\newcommand{\sign}{\operatorname{sign}}

\newcommand{\E}{{\mathbb E}}
\newcommand{\mb}{\mathbb}
\newcommand{\bs}{\mathbf}
\newcommand{\mc}{\mathcal}

\newcommand{\Var}{\operatorname{Var}}

\newcommand{\dotp}[2]{\left\langle#1,#2\right\rangle}

\newcommand{\pr}[1]{\mb{P}\left(#1\right)}

\providecommand{\sign}{\mathop\mathrm{sign}}

\newcommand\var{\mathrm{Var}}

\newcommand{\wt}{\widetilde}

\def\l{\left}
\def\r{\right}
\newcommand{\m}{\mathcal}

\newcommand{\W}{\mathbf{W}}
\newcommand{\V}{\mathbf{V}}
\newcommand{\A}{\mathbf{A}}
\newcommand{\B}{\mathbf{B}}
\newcommand{\R}{\mathbf{R}}
\newcommand{\T}{\mathbf{T}}

\title[Concentration inequalities for heavy-tailed random matrices]{Concentration and moment inequalities for sums of independent heavy-tailed random matrices}
\author{Moritz Jirak}
\address{Department of Statistics and Operations Research, Universit\"{a}t Wien, Austria}
\email{moritz.jirak@univie.ac.at}
\thanks{M. Jirak was supported by Austrian Science Fund (FWF), Project I
5485.}

\author{Stanislav Minsker}
\address{Department of Mathematics, University of Southern California, United States of America}
\email{minsker@usc.edu}
\thanks{S. Minsker and Y. Shen acknowledge support by the National Science Foundation grant DMS CAREER-2045068.}

\author{Yiqiu Shen}
\address{Department of Data Sciences and Operations, University of Southern California, United States of America}
\email{yiqiushe@usc.edu}

\author{Martin Wahl}
\address{Fakult\"{a}t f\"{u{}}r Mathematik, Universit\"{a}t Bielefeld, Germany}
\email{martin.wahl@math.uni-bielefeld.de}

\calclayout

\begin{document}
 \begin{abstract}
      We prove Fuk-Nagaev and Rosenthal-type inequalities for sums of independent random matrices, focusing on the situation when the norms of the matrices possess finite moments of only low orders. Our bounds depend on the ``intrinsic'' dimensional characteristics such as the effective rank, as opposed to the dimension of the ambient space. We illustrate the advantages of such results through several applications, including new moment inequalities for sample covariance matrices and their eigenvectors when the underlying distribution is heavy-tailed. Moreover, we demonstrate that our  techniques yield sharpened versions of moment inequalities for empirical processes.
\end{abstract}

\maketitle

\section{Introduction}

Fuk-Nagaev inequalities \citep{fuk1971probability,nagaev1979large} generalize exponential deviation inequalities for sums of independent random variables, such as Bernstein's, Prokhorov's and Bennett's inequalities, to the case when the random variables satisfy minimal integrability conditions. For example, a corollary of Fuk and Nagaev's results is the following bound: for a sequence of independent, centered random variables $X_1,\ldots,X_n$ such that $\max_k \mb E|X_k|^p<\infty$ for some $p\geq 2$ (here and below, $\mb E[\,\cdot\,]$ stands for the expectation and $\max_k$ is a shorthand for $\max_{k=1,\ldots,n}$),  
\begin{multline}
\label{eq:fuk-nagaev-version}
\pr{\Big| \sum_{k=1}^n X_k\Big| \geq t}\leq 2\exp\left(-C_1(p)\frac{t^2}{\sum_{k=1}^n \mb EX_k^2} \right) + \pr{\max_{k} |X_k| >t/4}
\\
+C_2(p)\left(\frac{\sum_{k=1}^n \mb E|X_k|^p}{t^p} \right)^2.
\end{multline}
\citet{nagaev1979large} describes the applications of such results to laws of large numbers and moment inequalities. 
Later, \citet{einmahl2008characterization}, \citet{adamczak2010few}, \citet{rio2017constants} and \citet{bakhshizadeh2023sharp}, among others, improved the original estimates by Fuk and Nagaev in several ways: first, the inequalities were extended to martingales and Banach-space valued random variables, and second, the constants were sharpened. For example, a version of inequality \eqref{eq:fuk-nagaev-version} due to \citet{rio2017constants} holds with $C_1(p)=1/(2+\delta)$ for any $\delta>0$ and $C_2(p)$ of order $p^p$ (for $p>2$). The latter fact is important as the order of growth of these constants translates into the tail behavior of $\left|\sum_{k=1}^n X_k\right|$. 

The goal of this work is to prove a version of the Fuk-Nagaev inequality for the sums of independent random matrices and use it to sharpen existing moment inequalities. 
Let $\W_1,\ldots,\W_n\in \mb C^{d\times d}$ be a sequence of independent self-adjoint\footnote{It is well-known that the case of general rectangular matrices reduces to this one via the so-called ``Hermitian dilation,'' see \cite[][section 2.1.17]{tropp2015introduction}.} random matrices such that $\mb E\W_k=\mathbf{0}_{d\times d}$ for all $k$ and where the expectation is taken element-wise. Assume that for all $k$, $\|\W_k\|\leq U$ with probability $1$ where $\|\cdot\|$ stands for the operator (spectral) norm. A line of work by \citet{ahlswede2002strong,oliveira2010sums,tropp2012user} culminated in the following version of the so-called matrix Bernstein's inequality: for all $t>0$, 
\[
\pr{\l\|\sum_{k=1}^n \W_k \r\|\geq t}\leq 2d \exp\l( -\frac{t^2/2}{\sigma^2 + Ut/3}\r)
\]
where $\sigma^2 = \l\| \sum_{k=1}^n \mb E \W_k^2\r\|$. An attractive feature of this inequality (as opposed to, say, Talagrand's concentration inequality \citep{talagrand1996new}) is that it yields a bound for $\mb E\|\sum_{k=1}^n \W_k \|$, namely, $\mb E\|\sum_{k=1}^n \W_k \|\leq K( \sigma\sqrt{\log(d)} + U\log(d))$ for some absolute constant $K>0$. Tropp's results have been extended in two directions: first, it was shown by \citet{tropp2015introduction,minsker2011some} that the dimension factor $d$ can essentially be replaced by the so-called \textit{effective rank} $r\big(\sum_{k=1}^n \mb E \W_k^2 \big)$ where 
\[
r(\mathbf{A}):=\frac{\mathrm{trace}(\mathbf{A})}{\|\mathbf{A}\|}
\]
for a nonnegative definite matrix $\mathbf{A}$ (with the convention that $r(\mathbf{0}_{d\times d}):=0$). In particular, this version of Bernstein's inequality is applicable in the context of random Hilbert-Schmidt operators acting on Hilbert spaces. Second, the boundedness assumption was relaxed by \citet{koltchinskii2011neumann} to the requirement that 
\(
\max_k \l\| \|\W_k\| \r\|_{\psi_1}<\infty
\)
where the $\psi_1$ norm of a random variable $Z$ is defined by \[
\|Z\|_{\psi_1} = \inf\{ r>0: \ \mb E e^{\l|Z\r|/r }\leq 2 \}.
\]
Finally, \citet{klochkov2020uniform} showed that there exist absolute constants $c_1,c_2>0$ such that 
\[
\pr{\l\|\sum_{k=1}^n \W_k \r\|\geq t}\leq c_1 r(\V_n^2) \exp\l( -c_2\frac{t^2}{\|\V_n^2\| + t\l\| \max_k \|\W_k\| \r\|_{\psi_1}}\r)
\]
for all $t\geq c_3(\|\V_n^2\|^{1/2}+\left\| \max_k \|\W_k\|
> \right\|_{\psi_1})$ where $\V_n^2$ is any matrix satisfying $\V_n^2\succeq \sum_{k=1}^n \mb E \W_k^2$ 
(here and below, for self-adjoint matrices $\mathbf{A}$ and $\mathbf{B}$, we write $\mathbf{A}\succeq \mathbf{B}$  or $\mathbf{B}\preceq \mathbf{A}$ if $\mathbf{A}-\mathbf{B}$ is positive semidefinite). Our results allow to relax the integrability assumptions even further and cover the case of heavy-tailed random matrices, namely, random matrices such that $\mb E\|\W\|^p = \infty$ for some $p>0$ (however, we are still able to recover known bounds for the ``light-tailed'' random matrices). For example, \Cref{th:rosenthal} below implies that for all $p\geq 1$,\footnote{Given a non-negative random variable $Z$, we will use a shorthand $\mb E^{1/p}\, Z^p$ to denote $\l(\mb E Z^p\r)^{1/p}$. } 
\[
\mb{E}^{1/p}\left\|\sum_{k=1}^n \W_k\right\|^p \leq K \left(\sqrt{q}\|\V_n^2\|^{1/2} +q\mb E\max_k \|\W_k\| + \frac{p}{\log (ep)} \mb{E}^{1/p} \max_k \|\W_k\|^p \right)
\]
where $q=\log(r(\V_n^2))\vee p$ and $K$ is an absolute constant. This inequality sharpens previously known results of this type by \citet{junge2013noncommutative,dirksen2011noncommutative,chen2012masked}. For example, a version of Rosenthal's inequality by \citet{chen2012masked} states that 
\[
\mb{E}^{1/p}\left\|\sum_{k=1}^n \W_k\right\|^p\leq K \left(\sqrt{\tilde q} \|\V_n^2\|^{1/2} + \tilde q\,\mb E^{1/p}\max_k \|\W_k\|^{p} \right)
\]
where $\tilde q = \log(d)\vee p$. The fact that our bound depends on $r(\V_n^2)$ instead of $d$ immediately allows to extend it to Hilbert-Schmidt operators acting on Hilbert spaces.  

Finally, let us remark that the order of the constants in the inequality stated above is optimal. In the scalar case ($d=1$), sharp versions of Rosenthal's inequality with explicit constants are known \cite[for example, see][]{figiel1997extremal,ibragimov1998exact,utev1985extremal,johnson1985best}. For instance, \citet{johnson1985best} showed that without any additional assumptions, the best order of $C(p)$ in the inequality
\[
\mb E^{1/p} \Big| \sum_{k=1}^n \W_k \Big|^p \leq C(p)\left( \left( \sum_{k=1}^n \mb E\W_k^2\right)^{1/2} + \left( \sum_{k=1}^n \mb E|\W_k|^p \right)^{1/p} \right) 
\]
is $C(p) = K\frac{p}{\log(p)}$, while \citet{pinelis1985estimates} proved that $C_1(p) = K\sqrt{p}$ and $C_2(p) = Kp$ are the best possible order in the inequality of the form 
\begin{equation}
\label{eq:rosenthal4}
\mb E^{1/p} \Big| \sum_{k=1}^n \W_k \Big|^p \leq C_1(p)\left( \sum_{k=1}^n \mb E\W_k^2\right)^{1/2} + C_2(p)\mb E^{1/p}\max_{k} |\W_k|^p.
\end{equation}
It is clear that our results in \Cref{th:rosenthal} below yield a sharper version of inequality \eqref{eq:rosenthal4} for large $p$ whenever $\mb E\max_k |\W_k|$ is much smaller than $\mb E^{1/p}\max\limits_k |\W_k|^p$.

\subsection{Organization of the paper}

The rest of the exposition is organized as follows: we present the main results and their proofs in Sections \ref{sec:fuk-nagaev} and \ref{sec:rosenthal}. Applications of the developed techniques to empirical processes is discussed in Section \ref{sec:empirical} while implications for the problems of matrix subsampling, covariance estimation and eigenvector estimation are described in Sections \ref{sec:subsampling}, \ref{sec:covariance} and \ref{section:examples:eigenvector} respectively. Finally, Appendix \ref{sec:tech} contains the necessary background material.

\section{Main results}

In this section we present new Fuk-Nagaev and Rosenthal-type inequalities (Theorems \ref{th:fuk-nagaev3} and \ref{th:rosenthal}, respectively). The general idea behind the proofs is as follows. After a standard symmetrization argument, the Fuk-Nagaev-type inequalities are derived by truncating the norms of random matrices $\W_1,\ldots,\W_n$. The bounded part is handled using a matrix Bernstein-type inequality, while the unbounded part is treated with several tools for sums of independent random vectors in Banach spaces -- notably, Hoffmann-J{\o}rgensen inequality and an $L_p$-$L_1$ estimate for sums of independent Banach space-valued random variables going back to Talagrand, and Kwapien and Szulga). The Rosenthal-type inequalities then follow via a tail integration argument.

The required notation will be  introduced on demand. Let us remark that throughout the paper, the values of constants $K$, $c$, $C(\cdot)$ is often unspecified and can change from line to line; we use $K$ and $c$ to denote absolute constants and $C(\cdot)$ to denote constants whose value depends on the parameters in brackets. 

\subsection{Fuk-Nagaev-type inequality}
\label{sec:fuk-nagaev}

The following proposition is the key technical ingredient that will serve as the starting point for the derivation of the main results. 
Everywhere below, $M=\max_{k}\|\W_k\|$, $Q_{1/2}(Z)$ stands for the median of a real-valued random variable $Z$, and $\eps_1,\ldots,\eps_n$ denote independent symmetric Bernoulli random variables that are independent from $\W_1,\ldots,\W_n$. Finally, $\|\cdot\|_2$ denotes the Euclidean norm and $\dotp{\cdot}{\cdot}$ denotes the Euclidean inner product. 
\begin{prop}
\label{th:fuk-nagaev}
Let $\W_1, \dots, \W_n\in \mb{C}^{d\times d}$ be a sequence of centered, independent, self-adjoint random matrices. Let $U>0$, and assume that $\V_n^2$ satisfies $\V_n^2\succeq \sum_k \mb{E} \W_k^2 \bs{1}\{\|\W_k\|\leq U\}$. Finally, set $\sigma_U^2=\|\V_n^2\|$. Then, whenever 
\begin{equation}
\label{eq:condition:t}
t \geq 4\sigma_U \vee 4U/3\vee \sup_{\|v\|_2=1} Q_{1/2}\l(\dotp{\l(\sum_{k=1}^n W_k \r)v}{v}\r),\end{equation}
the following inequality holds:
\begin{multline}
\label{eq:fuk-nagaev}
\mb{P}\left(\left\|\sum_{k=1}^n \W_k\right\| > 12t\right) \leq 64r\big(\V_n^2\big) \exp\left[-\frac{(t/2)^2}{2\l(\sigma_U^2 + tU/6\r)}\right] 
\\
+16\mb{P}\left(\left\|\sum_{k=1}^n \eps_k\W_k \bs{1}\{\|\W_k\| > U\} \right\|>t/2\right) \mb{P}\left(\left\|\sum_{k=1}^n \eps_k \W_k\right\|>t\right) + 4\mb{P}\left(M > t\right).
\end{multline}
If the random matrices are symmetrically distributed (that is, $\W_k$ and $-\W_k$ have the same distribution for all $k$), then 
\begin{multline}
\label{eq:fuk-nagaev2}
\mb{P}\left(\left\|\sum_{k=1}^n \W_k\right\| > 3t\right) \leq 16r\big(\V_n^2\big) \exp\left[-\frac{(t/2)^2}{2\l(\sigma_U^2 + tU/6\r) }\right] 
\\
+4\mb{P}\left(\left\|\sum_{k=1}^n \W_k \bs{1}\{\|\W_k\| > U\} \right\|>t/2\right) \mb{P}\left(\left\|\sum_{k=1}^n \W_k\right\|>t\right) + \mb{P}\left(M > t\right)
\end{multline}
under the assumption that $t\geq 4\sigma_U \vee 4 U/3$. 
\end{prop}
\begin{proof}
Let us reduce the general case to the situation where $\W_1,\ldots,\W_n$ are symmetrically distributed. To this end, it suffices to apply Lemma 2.3.7 in \citet{Vaart1996Weak-convergenc}: it implies that whenever $6t\geq \sup_{\|v\|_2=1} Q_{1/2}\l(\dotp{\l(\sum_{k=1}^n \W_k \r)v}{v}\r)$,
\[
\mb{P}\left(\left\|\sum_{k=1}^n \W_k\right\|>12t\right) \leq 4\mb{P}\left(\left\|\sum_{k=1}^n \eps_k \W_k\right\|>3t\right).
\]
Next, in view of Hoffmann-J{\o}rgensen's inequality (Lemma \ref{prop:HJtail} in Appendix \ref{sec:tech}), 
\[
\mb{P}\left(\left\|\sum_{k=1}^n \eps_k \W_k\right\|>3t\right) \leq 4 \mb{P}\left(\left\|\sum_{k=1}^n \eps_k \W_k\right\|>t\right)^2 + \mb{P}\left(M > t\right).
\]
Given $U>0$, we define, for each $k=1,\dots, n$, 
\begin{equation*}
    \wt{\mathbf{W}}_k:=\eps_k\W_k \bs{1}\{\|\W_k\|\leq U\} \text{ and } \bs{\Delta}_k:=\eps_k\W_k \bs{1}\{\|\W_k\| > U\}. 
\end{equation*}
Clearly, $\|\sum_{k=1}^n \eps_k\W_k\| \leq \|\sum_{k=1}^n \wt{\W}_k\| + \|\sum_{k=1}^n \bs{\Delta}_k\|$, all the random matrices $\wt{\W}_k$, $\bs{\Delta}_k$ are symmetric, and
\begin{equation*}
  \mb{P}\left(\left\|\sum_{k=1}^n \eps_k \W_k \right\|>t\right)\leq  \mb{P}\left(\left\|\sum_{k=1}^n \wt{\W}_k \right\|>t/2\right) + 
  \mb{P}\left(\left\|\sum_{k=1}^n \bs{\Delta}_k\right\|>t/2\right)
  :=D_1+D_2.
\end{equation*}
Therefore, 
\begin{equation}
\label{probSum}
\mb{P}\left(\left\|\sum_{k=1}^n \eps_k \W_k\right\|>3t\right) 
\leq 4D_1 + 4D_2 \mb{P}\left(\left\|\sum_{k=1}^n \eps_k \W_k\right\|>t\right) + \mb{P}\left(M > t\right).
\end{equation}
The first term on the right-hand side of inequality \eqref{probSum} can be bounded directly via a version of the matrix Bernstein's inequality (\Cref{matrixBernsteinErank} in Appendix \ref{sec:tech}). Let $\V_n^2$ satisfy $\V_n^2\succeq \sum_k \mb{E} \W_k^2 \bs{1}\{\|\W_k\|\leq U\}$. Then for $\sigma_U^2 = \|\V_n^2\|$ and $t$ such that 
$\frac{t}{2}\geq 2\l(\sigma_U \vee U/3 \r)$,
\begin{equation}
\label{eq:a02}
4D_1\leq 16 r\big(\V_n^2\big) \exp\left[-\frac{(t/2)^2}{2(\sigma_U^2 + tU/6)}\right].
\end{equation}
The result follows.
\end{proof}
\noindent We are now ready to deduce the first main result of the paper.  
\begin{theorem}
\label{th:fuk-nagaev3}
Let $\W_1, \dots, \W_n\in \mb{C}^{d\times d}$ be a sequence of centered, independent, self-adjoint random matrices, and assume that $\mb EM^p < \infty$ for some $p\geq 2$. Moreover, suppose that $\V_n^2$ satisfies $\V_n^2\succeq \sum_k \mb{E} \W_k^2$, and set $\sigma^2=\|\V_n^2\|$. Then
\begin{align}
\label{eq:fuk-nagaev3}
\mb{P}\left(\left\|\sum_{k=1}^n \W_k\right\| > 12t\right)&\leq  128r(\V_n^2) \exp\left[-\frac{(t/2)^2}{2(\sigma^2 + 4t\mb EM)}\right] 
 + 4\mb{P}\left(M\geq t\right) 
 \\
 &+ \left(\left(\frac{Kp}{\log (ep)}\right)^{p}\frac{\mb{E}M^p}{t^{p}}\right)^2
\end{align}
whenever $t\geq 4\sigma \vee 32\mb{E} M$ and where $K$ is an absolute constant. 
\end{theorem}
\begin{proof}
We will continue using the notation introduced in the proof of Proposition \ref{th:fuk-nagaev}. First of all, note that 
\begin{equation}
\label{eq:variance:1}    
\sup_{\|v\|_2=1} Q_{1/2}\l(\dotp{\l( \sum_{k=1}^n \W_k\r)v}{v}\r)\leq \sigma
\end{equation}
for the choice of $\sigma$ stated above. 
Indeed, for any centered random variable $Z$ with finite variance, $|Q_{1/2}(Z)|\leq \var^{1/2}(Z)$. 
Moreover, for any random self-adjoint matrix $\mathbf{Y}$ and any unit vector $v$, 
\[
\mb E\dotp{\mathbf{Y}v}{v}^2 = \mb E \l[ v^\ast \mathbf{Y} vv^\ast \mathbf{Y}v \r] \leq \mb E \l[ v^\ast \mathbf{Y}^2 v \r] \leq \| \mb E \mathbf{Y}^2 \|.
\]
Applying these inequalities to $\mathbf{Y} = \sum_{k=1}^n \W_k$ and $Z = \dotp{\mathbf{Y}v}{v}$, we deduce the relation~\eqref{eq:variance:1}. Next, plugging the inequality
\[
\mb{P}\left(\left\|\sum_{k=1}^n \eps_k \W_k \right\|>t\right)\leq  \mb{P}\left(\left\|\sum_{k=1}^n \wt{\W}_k \right\|>t/2\right) +  \mb{P}\left(\left\|\sum_{k=1}^n \bs{\Delta}_k\right\|>t/2\right)
\]
into bound \eqref{eq:fuk-nagaev} with $\V_n^2$ and $\sigma^2$ specified in the conditions of the theorem, we see that 
\begin{multline*}
    \mb{P}\left(\left\|\sum_{k=1}^n \W_k\right\|>12t\right) \leq 128r\big(\V_n^2\big) \exp\left[-\frac{(t/2)^2}{2(\sigma^2 + tU/6)}\right] + 4\mb{P}\left(M > t\right) 
    \\
    + 16 \l(\mb{P}\left(\left\|\sum_{k=1}^n \bs{\Delta}_k\right\|>t/2\right)\r)^2. 
\end{multline*}
Now we apply Markov's inequality to get the bound 
\[
\mb{P}\left(\left\|\sum_{k=1}^n \bs{\Delta}_k\right\|>t/2\right) \leq \frac{\mb{E}\left\|\sum_{k=1}^n \bs{\Delta}_k\right\|^{p}}{(t/2)^{p}}.
\]
\Cref{HJLp} implies that 
\begin{equation*}
    \mb{E}^{1/p}\left\|\sum_{k=1}^n \bs{\Delta}_k\right\|^{p} 
    \leq K \frac{p}{\log(ep)}\left(\mb{E}\left\|\sum_{k=1}^n \bs{\Delta}_k \right\|  + \mb{E}^{1/p}\max_{k}\left\|\bs{\Delta}_k \right\|^p  \right).
\end{equation*}
Moreover, if we set $U:=24 \mb EM$, then \Cref{HJexpectation} applies with $q=1$ and $t_0=0$, the latter due to the inequality
\begin{equation}
    \label{eq:U-bound}
    \mb{P}\left(\left\|\sum_{k=1}^n \bs{\Delta}_k\right\|>0\right)\leq \mb{P}\left(M>U\right)\leq \frac{\mb{E}M}{U}\leq \frac{1}{24}.
\end{equation}
Specifically, \Cref{HJexpectation} yields 
\begin{equation}
\label{eq:a00}
 \mb{E} \left\|\sum_{k=1}^n \bs{\Delta}_k\right\| \leq 6\,\mb{E}\max_{k}\|\bs{\Delta}_k\| 
 \leq 6\mb{E}\max_{k}\| \W_k\|  \leq 6\,\mb{E}^{1/p}\max_{k}\|\W_k\|^p, 
\end{equation}
implying that 
\begin{equation}
\label{eq:bound:delta:term}
\mb{P}\left(\left\|\sum_{k=1}^n \bs{\Delta}_k\right\|>t/2\right)\leq \l(K\frac{p}{\log(ep)} \r)^{p} \frac{\mb EM^p}{t^p}.
\end{equation}
The conclusion follows.
\end{proof}
\begin{remark}
\label{remark:rosenthal}
If $\W_1,\ldots,\W_n$ have symmetric distribution, then 
\[
\mb{P}\left(\left\|\sum_{k=1}^n \W_k\right\|>t\right) \geq \frac{1}{2}\pr{M>t}
\]
in view of L\'{e}vy's inequality \citep[][Proposition 2.3]{ledoux2013probability}. This shows that the term $\mb EM$ is necessary in the lower bound for $t$.  The term $\sigma$ is also known to be necessary. Take, for instance, $\W_k = \xi_k \A_k$ where $\xi_1,\ldots,\xi_n$ are i.i.d. $\mathrm{N}(0,1)$ random variables and $\A_1,\ldots,\A_n$ are fixed self-adjoint matrices \citep[][Section 4.1.2]{tropp2012user}.
\end{remark}
\begin{remark}
In principle, one can increase the power in the term $\left(\left(\frac{Kp}{\log (ep)}\right)^{p}\frac{\mb{E}M^p}{t^{p}}\right)^2$ on the right-hand side of \eqref{eq:fuk-nagaev3} from $2$ to $2^k$ by applying the Hoffmann-J{\o}rgensen's inequality $k$ times. This would require adjusting constants in the inequality, such as replacing multiplier $12$ on the left-hand side of \eqref{eq:fuk-nagaev3} by $3\cdot 2^{k+1}$.
\end{remark}

\subsection{Moment inequalities}
\label{sec:rosenthal}

Now we will establish moment inequalities by integrating the tail estimates of \Cref{th:fuk-nagaev}. 
As in the proof of Theorem \ref{th:fuk-nagaev3}, we set 
\[
M=\max_{k}\|\W_k\| \text{ and }\bs{\Delta}_k=\eps_k\W_k \bs{1}\{\|\W_k\| > 24\mb EM\}.
\]
Moreover, assume that $\V_n^2$ satisfies $\V_n^2\succeq \sum_k \mb{E} \W_k^2 \bs{1}\{\|\W_k\|\leq 24\mb EM\}$ and let $\sigma^2 = \|\V_n^2\|$. The frequent choice in applications is $\V_n^2 = \sum_k \mb{E} \W_k^2$.
\begin{theorem}
\label{th:rosenthal}
Let $\W_1, \dots, \W_n\in \mb{C}^{d\times d}$ be a sequence of centered, independent, self-adjoint random matrices, and 
let 
\[
R_p = \inf\{s>0: \  \mb{P}\left(\left\|\sum_{k=1}^n \bs{\Delta}_k\right\|>s/2\right)\leq \frac{1}{8}3^{-p}\}.
\]
Then for all $p \geq 1$, 
\begin{equation}
\label{eq:rosenthal}
\mb{E}^{1/p}\left\|\sum_{k=1}^n \W_k\right\|^p\leq K \left(\sqrt{q}\sigma +q \mb EM + R_p + \mb E^{1/p} M^p \right),
\end{equation}
where $q=\log(r(\V_n^2))\vee p$ and $K>0$ is an absolute constant. In particular, we have the following ``closed-form'' Rosenthal-type moment inequalities: 
\begin{align}
\label{ineq:rosenthal1}
\mb{E}^{1/p}\left\|\sum_{k=1}^n \W_k\right\|^p &\leq K \left(\sqrt{q}\sigma +q\mb EM + \frac{p}{\log (ep)} \mb{E}^{1/p} M^p \right) \text{ and }
\\
\label{ineq:rosenthal2}
\mb{E}^{1/p}\left\|\sum_{k=1}^n \W_k\right\|^p &\leq K \left(\sqrt{q}\sigma +\log(r(\V_n^2))\mb EM + p\|M\|_{\psi_1}\right).
\end{align}
\end{theorem}

\begin{remark}
    The well-known relation $\|M\|_{\psi_1}\leq K\log(n)\max_{k=1,\ldots,n}\l\|\|W_k\| \r\|_{\psi_1}$ (see \citep[]{Vaart1996Weak-convergenc}) could be useful when combined with inequality \eqref{ineq:rosenthal2}. 
\end{remark}
\begin{proof}
Observe that 
\begin{multline}
\label{eq:b01}
\mb E \l\| \sum_{k=1}^n \W_k\r\|^p \leq 2^p \mb E \l\| \sum_{k=1}^n \eps_k \W_k\r\|^p 
\\
= 2^p p\int_0^\infty t^{p-1}\pr{\l\| \sum_{k=1}^n \eps_k \W_k\r\| \geq t}dt  
= 6^p p\int_0^\infty t^{p-1}\pr{\l\| \sum_{k=1}^n \eps_k \W_k\r\| \geq 3t}dt,
\end{multline}
where we used the symmetrization inequality \citep[][Lemma 2.3.6]{Vaart1996Weak-convergenc} in the first step, the integration by parts formula in the second step, and the linear change of variables in the third step. To estimate the last integral, we set $U := 24 \mb EM$ and apply inequality \eqref{eq:fuk-nagaev2} in the range $t\geq t_0:= 4(\sigma\vee U/3)$ to deduce that 
\begin{multline}
\label{eq:moment:bound:2}
2^p \mb E \l\| \sum_{k=1}^n \eps_k \W_k\r\|^p 
\leq 24^p(\sigma\vee U/3)^p 
\\
+ 6^p p\Bigg( \int_0^\infty t^{p-1}\l(16r\left(\V_n^2 \right) \exp\left[-\frac{(t/2)^2}{2(\sigma^2 + tU/6)}\right] \wedge 1\r)dt 
\\
+ \int_0^\infty t^{p-1} \pr{M\geq t} dt
+ \int_0^\infty 4t^{p-1}\mb{P}\left(\left\|\sum_{k=1}^n \bs{\Delta}_k\right\|>t/2\right)\mb{P}\left(\left\|\sum_{k=1}^n \eps_k \W_k\right\|>t\right) dt \Bigg).
\end{multline}
Recalling the definition of $R_p$, one easily checks that 
\begin{multline*}
    6^p p\int_0^{\infty} 4t^{p-1}\mb{P}\left(\left\|\sum_{k=1}^n \bs{\Delta}_k\right\|>t/2\right)\mb{P}\left(\left\|\sum_{k=1}^n \eps_k \W_k\right\|>t\right) dt 
    \\
    \leq 4 \cdot 6^p R_p^p + 2^{p-1} \mb E \l\| \sum_{k=1}^n \eps_k \W_k\r\|^p. 
\end{multline*}
Combined with \eqref{eq:moment:bound:2}, this inequality implies that 
\begin{multline}
\label{eq:ineq:rosenthal:2}
2^{p-1} \mb E \l\| \sum_{k=1}^n \eps_k \W_k\r\|^p 
\leq 24^p(\sigma\vee U/3)^p  + 6^p \mb EM^p + 4 \cdot 6^p R_p^p
\\ 
+ 6^p p \int_0^\infty t^{p-1}\l(16r\left(\V_n^2\right) \exp\left[-\frac{(t/2)^2}{\sigma^2 + tU/6}\right] \wedge 1\r)dt.
\end{multline}
Arguing as in the proof of Corollary \ref{cor:matrixBernsteinLp} in Appendix \ref{sec:tech} implies that the last integral in the above display does not exceed
\[
K^p\l( \sqrt{q}\sigma + q U \r)^p,
\]
where $q=\log(r(\V_n^2))\vee p$ and $K>0$ is an absolute constant.
Combined with the fact that $U = 24\mb EM$, this yields \eqref{eq:rosenthal} of the theorem.

Finally, let us prove the inequalities \eqref{ineq:rosenthal1} and \eqref{ineq:rosenthal2}. To this end, we need to obtain the upper bound for the quantile $R_p$. 
Note that in view of the estimate \eqref{eq:bound:delta:term}, $R_p \leq K \frac{p}{\log(ep)} \mb E^{1/p}M^p$. This establishes inequality \eqref{ineq:rosenthal1}; let us remark that it can also be obtained by directly integrating the tail bound \eqref{eq:fuk-nagaev3}. 
If, on the other hand, $\|M\|_{\psi_1}<\infty$, then the second inequality of Lemma \ref{HJLp} in Appendix \ref{sec:tech} combined with the bound \eqref{eq:a00} and the well-known estimate $\mb E M \leq K \|M\|_{\psi_1}$ imply that $\l\| \sum_{k=1}^n \bs{\Delta}_k\r\|_{\psi_1}\leq K\l\| M\r\|_{\psi_1}$,  
thus 
\[
\pr{\left\|\sum_{k=1}^n \bs{\Delta}_k\right\|>t} \leq e^{-c\frac{t}{\|M\|_{\psi_1}}}
\]
and $R_p \leq c' p\|M\|_{\psi_1}$. Moreover, since $\mb E^{1/p} M^p \leq K p\|M\|_{\psi_1}$, inequality \eqref{ineq:rosenthal2} follows. \end{proof}

\noindent Next, we deduce a version of the previous result that holds for sums of nonnegative definite random matrices.

\begin{corollary}
\label{th:rosenthal2}
Let $\W_1,\ldots,\W_n\in \mb C^{d\times d}$ be a sequence of independent, self-adjoint, nonnegative definite matrices, and let $M=\max_{k}\|\W_k\|$. Moreover, let $\A_n:=\sum_{k=1}^n \mb E \W_k$. Then for all $p\geq 1$, 
\begin{align*}
\mb{E}^{1/p}\left\|\sum_{k=1}^n \W_k\right\|^p
&\leq K \left(\left\| \A_n \right\| +q\,\mb{E} M + \frac{p}{\log (ep)} \mb{E}^{1/p} M^p
\right),
\\
\mb{E}^{1/p}\left\|\sum_{k=1}^n \W_k\right\|^p
&\leq K \bigg(\left\| \A_n \right\| +\log(r(\A_n))\mb{E} M +p\|M\|_{\psi_1}
\bigg),
\end{align*}
where $q=\log(r(\A_n))\vee p$ and $K>0$ is an absolute constant.  
\end{corollary}
\begin{proof}  
In view of Minkowski's inequality followed by the symmetrization inequality, 
\begin{align*}
\mb{E}^{1/p}\left\|\sum_{k=1}^n \W_k\right\|^p&\leq 
\|\A_n\| + \mb{E}^{1/p}\left\|\sum_{k=1}^n \W_k -\mb E\W_k\right\|^p
\\
&\leq \|\A_n\| + 2\mb{E}^{1/p}\left\|\sum_{k=1}^n \eps_k \W_k \right\|^p,
\end{align*}
where $\eps_1,\ldots,\eps_n$ are i.i.d.~random signs independent of $\W_1,\ldots,\W_n$. 
To estimate the second term in the sum above, we apply inequality \eqref{ineq:rosenthal1}. Recall that $\V_n^2$ must satisfy $\V_n^2 \succeq \sum_k \mb{E} \W_k^2 \bs{1}\{\|\W_k\|\leq 24\mb EM\}$. Note that for all $k$,
\[
\mb E \l[ \W_k^2 \bs{1}\{ \|\W_k \| \leq 24 \mb EM \} \r]
\preceq 24 \mb EM\cdot \mb E\W_k
\]
since $\W_k \succeq 0$ with probability 1. This relation implies that we can set 
\[
\V_n^2 = 24 \mb E M\cdot \A_n,
\]
whence $r(\V_n^2) = r(\A_n)$. Moreover, 
\(\sigma\sqrt{q} = \sqrt{24q \|\A_n\| \mb EM} 
\leq \|\A_n\| + 6q \mb EM,
\)
hence \eqref{ineq:rosenthal1} yields the bound
\[
\mb{E}^{1/p}\left\|\sum_{k=1}^n \eps_k \W_k \right\|^p \leq 
K'\l( \|\A_n\| + q\mb EM + \frac{p}{\log(ep)} \mb E^{1/p} M^p\r),
\]
which implies the claim. The second inequality is obtained in a similar manner where the inequality \eqref{ineq:rosenthal2} is used in place of \eqref{ineq:rosenthal1}.  
\end{proof}

\subsection{Inequalities for empirical processes}
\label{sec:empirical}

The only part of the previous arguments that exploits the ``non-commutative'' nature of the random variables is the application of Matrix Bernstein's inequality. In this section, we state the results produced by our method for general empirical processes. The only required modification is the application of Bousquet's version of Talagrand's concentration inequality \eqref{eq:bousquet} in place of Bernstein's inequality. We state only the versions of Theorems \ref{th:fuk-nagaev3} and \ref{th:rosenthal} and remark on the key differences. The required changes to the proofs are minimal, and hence we avoid the details. 

Let $\mc{F}$ be a set of measurable real-valued functions defined on some measurable space $S$ and let $X_1,\ldots,X_n$ be i.i.d. copies of an $S$-valued random variable $X$. Assume that $\mb{E}f(X)=0$ for all $f\in \m F$. Set $F(x):=\sup_{f\in \mc{F}}|f(x)|$, $M = \max_{k=1,\ldots, n} F(X_k)$, and suppose that $\mb EM^p<\infty$ for some $p\geq 2$. Denote $Z=\sup _{f \in \mathcal{F}} \sum_{k=1}^n f(X_k)$; for simplicity, we assume that $Z$ is measurable. 
Finally, let $\sigma_\ast$ satisfy $\sigma_\ast^2\geq n\sup_{f\in \mc{F}} \mb{E}f^2(X)$. For example, in the main case of interest in this article, $\l\| \sum_{k=1}^n \W_k \r\| = \sup_{\|v\|_2=1} \l|\mathrm{trace}\l( (\sum_{k=1}^n \W_k) vv^{\top}\r) \r|$ corresponding to 
\(
\mc{F} = \l\{ f_v(\cdot) = \l|\mathrm{trace}\l((\cdot)vv^{\top}\r)\r|, \ \|v\|_2=1\r\}.
\)
The following result can be viewed as an improvement of Adamcszak's inequality \cite[][Theorem 2]{adamczak2010few} which in turn extends Adamczak's result in \citep{adamczak2008tail} to the heavy-tailed case.
\begin{theorem}
\label{th:emp1}
For all $t\geq \sqrt{2}\sigma_\ast$,
\begin{multline}
\label{eq:talagrand:2}
\mb{P}\left(Z > 24\l( \mb EZ + t\r)\right)\leq  K\Bigg( \exp\left(-\frac{t^2}{2\sigma_\ast^2 + 64t\mb EM}\right) 
 + \mb{P}\left(M\geq t\right) 
 \\
 + \left(\frac{p}{\log (ep)}\right)^{2p}\left(\frac{\mb{E}M^p}{t^{p}}\right)^2\Bigg)
\end{multline}
where $K$ is an absolute constant. 
\end{theorem}
\noindent Adamczak's version of the bound states that for any $\delta>0$, $0<\eta\leq 1$, $p\geq 1$, there exists $C(\delta,\eta,p)>0$ such that for all $t\geq 0$,
\begin{equation}
    \label{eq:adamczak:1}
\mb{P}\l( Z \geq (1+\eta)\mb EZ + t\r)\leq \exp\left(-\frac{t^2}{2(1+\delta)\sigma^2}\right) + C(\delta,\eta,p)\frac{\mb EM^p}{t^p}.
\end{equation}
Note that, compared to \eqref{eq:talagrand:2}, tail bounds given by inequality \eqref{eq:adamczak:1} are slightly worse. 
For instance, integrating inequality \eqref{eq:talagrand:2} yields the following moment bound, while \eqref{eq:adamczak:1} would not be sufficient for this purpose. 
\begin{theorem}
\label{th:emp2}
For all $p\geq 1$,
\begin{equation}
    \label{eq:emp2}
\mb{E}^{1/p} Z^p \leq K \left(\mb EZ + \sigma_\ast \sqrt{p}+p\mb EM + \frac{p}{\log (ep)} \mb{E}^{1/p} M^p \right).
\end{equation}
\end{theorem}
Let us compare this result with the bound of Theorem \ref{th:rosenthal}. 
When applied to the sums of random matrices, we get the inequality
\[
\mb{E}^{1/p} \left\|\sum_{k=1}^n \W_k\right\|^p \leq K \left(\mb{E}\left\|\sum_{k=1}^n \W_k\right\| + \sigma_\ast \sqrt{p}+p\mb EM + \frac{p}{\log (ep)} \mb{E}^{1/p} M^p \right)
\]
where $\sigma^2_\ast = n\sup_{\|v\|_2=1} \mb E \l( v^{\top}\W_1 v\r)^2$. The main difference is that this bound includes $\mb{E}\left\|\sum_{k=1}^n \W_k\right\|$ on the right-hand side. However, for large values of $p$, it is better than \eqref{ineq:rosenthal1} since $\sigma^2_\ast$ can be much smaller than $\l\|\sum_k \mb{E} \W_k^2 \r\|$. Moreover, Theorem \ref{th:emp2} improves on the moment inequality established by \citet{gine2000exponential} (also see \citep{boucheron2005moment}): for instance, Proposition 3.1 in \cite{gine2000exponential} states that for all $p\geq 2$,
\[
\mb{E}^{1/p} Z^p\leq K \left(\mb EZ + \sigma_\ast \sqrt{p}+ p \mb{E}^{1/p} M^p \right).
\]
The estimate provided by \eqref{eq:emp2} is better for large values of $p$ if $\mb EM$ is much smaller than $\mb E^{1/p} M^p$.

\section{Applications}
\label{sec:applications}

In this section, we apply the inequalities to get improved bounds to three classical problems - matrix subsampling, covariance estimation, and empirical eigenvector estimation.

\subsection{Norms of random submatrices}
\label{sec:subsampling}

Let $\B$ be a self-adjoint matrix, and let $\delta_1,\ldots,\delta_d$ be i.i.d.~Bernoulli random variables with $\mb E\delta_1 = \delta \in (0,1)$. Define $\R = \mathrm{diag}(\delta_1,\ldots,\delta_d)$. We are interested in the spectral norm of the matrix $\B\R$ formed by the columns $B_i, \ i\in I$ of $\B$ with indices corresponding to the random set $I=\{ 1\leq i\leq d: \ \delta_i = 1\}$. 
This problem has previously been studied by \citet{rudelson2007sampling} and \citet{tropp2015introduction} who showed that 
\begin{equation}
    \label{eq:vershynin1}
\mb E \|\B\R\|^2 \leq K\l( \delta \|\B\|^2 + \frac{\log(n\delta)}{\lfloor \delta^{-1}\rfloor} \sum_{k=1}^{\lfloor \delta^{-1}\rfloor} \l\|B_{(k)}\r\|_2^2 \r)
\end{equation}
and 
\[
\mb E \|\B\R\|^2 \leq 1.72\l( \delta \|\B\|^2 + \log\l(2\frac{\|\B\|^2_{\mathrm{F}}}{\|\B\|^2}\r) \l\|B_{(1)}\r\|_2^2 \r)
\]
respectively, where $B_{(j)}$ denotes the column of $\B$ with the $j$-th largest norm and $K$ is a numerical constant. Note that 
$\l\| B_{(1)}\r\|_2^2\geq \frac{1}{\lfloor \delta^{-1}\rfloor} \sum_{k=1}^{\lfloor \delta^{-1}\rfloor} \l\| B_{(k)}\r\|_2^2 $ but it is possible that $\log(n\delta) > \log\l(2\frac{\|\B\|^2_{\mathrm{F}}}{\|\B\|^2}\r)$ when the matrix $\B$ has small ``stable rank'' $\mathrm{srank}(\B) := \frac{\|\B\|^2_{\mathrm{F}}}{\|\B\|^2}$. 
We will show below that Tropp's bound can be improved, and that \eqref{eq:vershynin1} holds with $\log(n\delta)$ replaced with $\log(n\delta)\wedge \log(\mathrm{srank}(\B))$. To this end, let $e_1,\ldots,e_d$ denote the standard Euclidean basis, and observe that 
\[
\|\B\R\|^2= \l\| \sum_{k=1}^d \delta_k B_k e_k^{\top} \r\|^2
=\l\| \sum_{k=1}^d \delta_k B_k B_k^{\top}\r\|.
\]
We will apply \Cref{th:rosenthal2} to the last expression with $\W_k = \delta_k B_k B_k^{\top}$. Note that 
$\A_n = \delta \B\B^{\top}$ so that $\|\A_n\|=\delta \|\B\|^2$, and that $\mb E M = \mb E \l(\max_{k}\delta_k \|B_k\|_2^{2}\r)$. 
According to Lemma \ref{lemma:max-expectation} in Appendix \ref{sec:tech},
\[
\mb E\max_{k=1,\ldots,d}\delta_k \|B_k\|_2^2 \leq \frac{2}{\lfloor \delta^{-1}\rfloor}\sum_{k=1}^{\lfloor \delta^{-1}\rfloor} \|B_{(k)}\|_2^2.
\]
Finally, the effective rank $r(\A_n) = \frac{\|\B\|^2_{\mathrm{F}}}{\|\B\|^2}$ coincides with the stable rank of $B$. We record the following bound. 
\begin{corollary}
\label{cor:subsampling}
The inequalities  
\[
\mb E \|\B\R\|^2 \leq K\l( \delta \|\B\|^2 + \log\l(\frac{\|\B\|^2_{\mathrm{F}}}{\|\B\|^2}\r) \frac{1}{\lfloor \delta^{-1}\rfloor} \sum_{k=1}^{\lfloor \delta^{-1}\rfloor} \l\|B_{(k)}\r\|_2^2 \r)
\]
and 
\[
\mb E \|\B\R - \delta \B\|^2 \leq K(1-\delta)\l( \delta \|\B\|^2 + \log\l(\frac{\|\B\|^2_{\mathrm{F}}}{\|\B\|^2}\r) \frac{1}{\lfloor \delta^{-1}\rfloor} \sum_{k=1}^{\lfloor \delta^{-1}\rfloor} \l\|B_{(k)}\r\|_2^2 \r)
\]
hold for all $\delta\in(0,1)$ and a numerical constant $K>0$.
\end{corollary}
\begin{proof}
The first inequality has already been established above. The proof of the second bound is quite similar: it suffices to note that 
\[
\mb E \|\B\R - \delta \B\|^2 =\mb E\l\| \sum_{k=1}^d (\delta_k-\delta)^2 B_k B_k^{\top}\r\|,
\]
and that $\l\|\mb E \sum_{k=1}^d (\delta_k-\delta)^2 B_k B_k^{\top} \r\|= \delta(1-\delta) \|\B\|^2$. Moreover, 
\[
\mb E\max\limits_{k=1,\ldots,d}(\delta_k-\delta)^2\|B_k\|_2^2 \leq \frac{2(1-\delta)}{\lfloor \delta^{-1}\rfloor}\sum_{k=1}^{\lfloor \delta^{-1}\rfloor} \|B_{(k)}\|_2^2.
\]
which follows from Lemma \ref{lemma:max-expectation} in Appendix \ref{sec:tech}.
\end{proof}   

\subsection{Covariance estimation}
\label{sec:covariance}

In this section, we consider applications of our results to the covariance estimation problem. Let $X\in \mb R^d$ be a random vector such that $\mb{E}X=0$ and $\mb{E} XX^\top = \bs{\Sigma}$. Given a sequence $X_1,\dots, X_n\in \mb{R}^d$ of i.i.d.~copies of $X$, what is an upper bound for the error of the sample covariance matrix? In other words, we would like to estimate $\m R_n:=\l\| \frac{1}{n}\sum_{k=1}^n X_k X_k^\top -\bs{\Sigma}\r\|$. One of the long-standing open questions asks for the minimal assumptions on the distribution of $X$ such that $n=C(\eps)d$ observations suffices to guarantee that $\mb E \m R_n\leq \eps\|\bs{\Sigma}\|$, or that $\m R_n \leq \eps\|\bs{\Sigma}\|$ with high probability. Results of this type are often referred to as ``quantitative versions of the Bai-Yin theorem,'' after \citet{bai2008limit}. Let us give an (incomplete) overview of the rich history of the problem. It has long been known that $n=C(\eps)d$ observations suffice to get desired error bounds when the underlying distribution is sub-Gaussian \citep{vershynin2012close}. Moreover, very general and precise characterization of the behavior of the sample covariance of Gaussian random vectors with values in Banach spaces has been found by \citet{koltchinskii2017concentration} and, very recently, further sharpened by \citet{han2022exact} in the finite-dimensional case. For log-concave and sub-exponential distributions, the problem was first considered by \citet{kannan1997random}, and the bounds were significantly improved and refined by \citet{bourgain1996random} and \citet{rudelson1999random}. It took much longer to eliminate the unnecessary logarithmic factors, until the problem was finally solved by \citet{adamczak2010quantitative,adamczak2011sharp}. Finally, the case of heavy-tailed distributions was investigated by \citet{vershynin2011approximating,srivastava2013covariance,mendelson2014singular,guedon2017interval} and \citet{tikhomirov2018sample}, who showed that $4+\eps$ moments are sufficient to get the desired bound. Specifically, Tikhomirov's results imply that if $\bs{\Sigma} = I_d$ and $\sup_{\|v\|_2=1} \mb E\l|\langle X,v \rangle\r|^p = T<\infty$ for some $p>4$, then 
\[
\m R_n \leq C(p)\l( T^{2/p}\sqrt{\frac{d}{n}} + \frac{\max_{k}\|X_k\|_2^2}{n}\r)
\]
with probability at least $1-1/n$. \citet{abdalla2022covariance} refined Tikhomirov's estimates and essentially showed that $n = C(\eps)r(\bs{\Sigma})$ samples suffice to get the desired guarantees in expectation, although they considered the sample covariance based on properly truncated random vectors. Next, we show that the results by \citet{abdalla2022covariance} can be combined with the moment inequalities developed in this paper to get a sharp moment inequality for $r_n$. 
\begin{theorem}
\label{thm:covariance:estimation}
Let $X\in \mb R^d$ be a random vector such that $\mb{E}X=0$ and $\mb{E} XX^\top = \bs{\bs{\Sigma}}$. Let $X_1,\dots, X_n\in \mb{R}^d$ be i.i.d.~copies of $X$. Assume that $\frac{r(\bs{\Sigma})}{n}\leq c$ for a sufficiently small constant $c>0$, and that for some $p>4$ 
\begin{equation}
\label{eq:hypercontract}
\sup_{\|v\|_2=1}\frac{\mb{E}^{1/p}\l| \langle X,v\rangle \r|^p}{\mb{E}^{1/2}\langle X,v\rangle^2} = \kappa < \infty.
\end{equation}
Then  
\[
\mb{E}^{1/2}\left\|\frac{1}{n}\sum_{k=1}^n X_k X_k^\top -\bs{\Sigma} \right\|^2
\leq C(\kappa,p) \Bigg( \|\bs{\Sigma}\|\sqrt{\frac{r(\bs{\Sigma})}{n}}
+ \frac{\mb{E}^{2/p}\max_{k}\|X_k\|_2^p}{n}\Bigg).
\]
\end{theorem}
\begin{proof}
As stated above, our proof builds on the results of \citet{abdalla2022covariance}, which in turn sharpen the inequality due to \citet{tikhomirov2018sample}. Before we dive into the details, let us mention that the ``hypercontractivity'' condition \eqref{eq:hypercontract} implies, in particular, that 
\begin{equation}
    \label{eq:4th:moment}
\mb E \l( XX^{\top}\r)^2 =  \mb E \l[\|X\|_2^2 XX^{\top}\r] \preceq \kappa^4 \mathrm{trace}(\bs{\Sigma})\bs{\Sigma}. 
\end{equation}
Inequality \eqref{eq:4th:moment} is well known and follows, for instance, from the proof of Lemma 2.3 in \citep{wei2017estimation}. Note that 
\[
\left\|\frac{1}{n}\sum_{k=1}^n X_k X_k^\top -\bs{\Sigma} \right\| = \sup_{\|v\|_2=1}\left|\frac{1}{n}\sum_{k=1}^n \langle X_k, v\rangle^2 - \mb{E}\langle X,v\rangle^2\right|.
\]
Let us state the following decomposition of the error into ``peaky'' and ``spread'' parts (following \citep[][]{tikhomirov2018sample}) that holds for arbitrary $\lambda>0$: 
\begin{multline}
\label{peaky-spreadsplit}
\sup_{\|v\|_2= 1}\left|\frac{1}{n}\sum_{k=1}^n \langle X_k, v\rangle^2 - \mb{E}\langle X,v\rangle^2\right| 
\leq 
\underbrace{\sup_{\|v\|_2 = 1} \frac{1}{n}\sum_{k=1}^n \langle X_k, v\rangle^2 \mathbf{1}\{\lambda \langle X_k, v\rangle^2 >1 \}}_{\text{Peaky part}} \\ + \underbrace{\sup_{\|v\|_2 = 1} \left|\frac{1}{\lambda n}\sum_{k=1}^n \psi(\lambda\langle X_k, v\rangle^2 )- \mb{E}\langle X,v\rangle^2\right|}_{\text{Spread part}},
\end{multline}
where
\[
    \psi(x)=\begin{cases}
        x, & \text{for } x\in [-1,1], \\
        \sign(x) & \text{for } |x|>1.
    \end{cases}
\]
We will estimate the two terms separately, starting with the ``spread'' part. 
To this end, we apply Proposition 3.1 in \citep[][]{abdalla2022covariance} which implies that for $\lambda = \frac{1}{\kappa^2\|\bs{\Sigma}\|}\sqrt{\frac{r(\bs{\Sigma})}{n}}$,
\begin{equation*}
\sup_{\|v\|_2 = 1}\left|\frac{1}{\lambda n}\sum_{k=1}^n \psi(\lambda\langle X_k, v\rangle^2 )- \mb{E}\langle X,v\rangle^2\right|
\leq C\kappa^2\|\bs{\Sigma}\|\l( \sqrt{\frac{r(\bs{\Sigma})}{n}} + \frac{t}{\sqrt{r(\bs{\Sigma})n}}\r)
\end{equation*}
with probability at least $1-e^{-t}$. For $t=r(\bs{\Sigma})$, we get in particular that
\begin{equation}
\label{spreadfinalbound}  
\sup_{\|v\|_2 = 1}\left|\frac{1}{\lambda n}\sum_{k=1}^n \psi(\lambda\langle X_k, v\rangle^2 )- \mb{E}\langle X,v\rangle^2\right|
\leq C\kappa^2\|\bs{\Sigma}\|\sqrt{\frac{r(\bs{\Sigma})}{n}}
\end{equation}
with probability at least $1-e^{-r(\bs{\Sigma})}$. Next, we will estimate the ``peaky'' term in inequality \eqref{peaky-spreadsplit}. Identity (5) in \cite{abdalla2022covariance} implies that for all subsets $J\subseteq [n]:=\{1,\ldots,n\}$ of cardinality at most $l$,
\begin{equation}
\label{partialsum}
    \left\|\frac{1}{n}\sum_{k\in J} X_k X_k^\top \right\| \leq \frac{f(l, [n])}{n},
\end{equation}
where the function $f(l, [n])$ is defined by
\begin{equation*}
    f(l,[n]) = \sup_{\|y\|_2=1, \|y\|_0\leq l }\left\|\sum_{k=1}^n y_k X_k\right\|_2^2.
\end{equation*}
Here, $\|y\|_0$ stands for $\sum_{k=1}^n I\{y_k \ne 0\}$, $y\in \mb R^n$. The following result provides an upper bound for $f(l,[n])$. 
\begin{theorem}\cite[Theorem 3]{abdalla2022covariance} 
\label{fbound}
Assume that $\frac{r(\bs{\Sigma})}{n}\leq c'$ for a sufficiently small positive constant $c'$. Then 
    \begin{equation*}
        f(l,[n])\leq C(p,\kappa)\left(\max_{k=1,\ldots,n} \|X_k\|_2^2 + \|\bs{\Sigma}\| l \left(\frac{n}{l}\right)^{4/(4+p)} \log^4 \frac{n}{l}\right),
    \end{equation*}
with probability at least $1-\frac{C(p)}{n}$, and the bound holds simultaneously for all integers $l$ satisfying $r(\bs{\Sigma}) \leq l \leq c'n$.
\end{theorem}
\noindent  Since $p>4$, this result implies that with probability at least $1-\frac{C(p)}{n}$,
\begin{equation}
\label{fsbound}
    f(l,[n])\leq C'(p,\kappa)\left(\max_{k} \|X_k\|_2^2 + \|\bs{\Sigma}\|  \sqrt{nl}\right).
\end{equation}
Next, let  
\[
\lambda = \frac{1}{\kappa^2\|\bs{\Sigma}\|}\sqrt{\frac{r(\bs{\Sigma})}{n}}.
\]
Following \cite{tikhomirov2018sample,abdalla2022covariance}, let us define the random set $I_v$ via
\[ 
I_v:=\l\{k\in [n]: \langle X_k, v\rangle^2 > 1/\lambda\r\} 
\] 
and set 
\[
m:=\sup_{\|v\|_2 = 1} |{I}_v|.
\]
Then, in view of inequality \eqref{partialsum}, we see that
\begin{equation}
\label{peakyfbound}
    \frac{m}{n\lambda}\leq \sup_{\|v\|_2 = 1} \frac{1}{n}\sum_{k=1}^n \langle X_k, v\rangle^2 \mathbf{1}\{\lambda \langle X_k, v\rangle^2 >1 \}\leq \frac{f(m,[n])}{n}.
\end{equation}
Now, if $r(\bs{\Sigma})\leq m \leq c'n$, then we can employ inequality \eqref{fsbound} to derive the following bound that holds with probability at least $1-\frac{C(p)}{n}$: 
\begin{multline*}
    m\leq \lambda f(m,[n]) \leq C'(p,\kappa)\left(\max_{k} \|X_k\|_2^2 + \|\bs{\Sigma}\|  \sqrt{nm}\right)\cdot \frac{1}{\kappa^2\|\bs{\Sigma}\|}\sqrt{\frac{r(\bs{\Sigma})}{n}}
    \\ \leq C'(p,\kappa)\left(\frac{\max_{k}\|X_k\|^2}{\|\bs{\Sigma}\|}\sqrt{\frac{r(\bs{\Sigma})}{n}} + \sqrt{m} \sqrt{r(\bs{\Sigma})}\right). 
\end{multline*}
Solutions to the inequality $x \leq a\sqrt{x} + b$ satisfy $x\leq 2\max(a^2, b)$. Therefore, with probability at least $1-\frac{C(p)}{n}$, 
\begin{equation}
\label{mbound}
    m\leq C_1(p,\kappa)\max\left(r(\bs{\Sigma}), \frac{\max_{k}\|X_k\|^2}{\|\bs{\Sigma}\|}\sqrt{\frac{r(\bs{\Sigma})}{n}}\right). 
\end{equation}
It remains to show that with high probability, $m\leq c'n$ if $r(\bs{\Sigma})<cn$ for $c$ small enough (clearly, if $m<r(\bs{\Sigma})$, then \eqref{mbound} holds). 
For brevity, set
\[
Z_k(v) := \frac{|\langle X_k, v\rangle|}{\left({\kappa^2\|\bs{\Sigma}\|}\sqrt{\frac{n}{r(\bs{\Sigma})}}\right)^{1/2}}\quad\text{and}\quad S=\sup_{\|v\|_2 = 1}\left(\sum_{k=1}^n \rho(Z_k(v)) - \mb{E} \rho(Z_k(v))\right). 
\]
By the definition of $I_v$, for any $v\in \mb{R}^d$,
\begin{multline}
        |I_v|  = \sum_{k=1}^n\mathbf{1}\left\{\langle X_k, v\rangle^2 > {\kappa^2\|\bs{\Sigma}\|}\sqrt{\frac{n}{r(\bs{\Sigma})}}\right\} 
        = \sum_{k=1}^n \mathbf{1}\left\{Z_k(v) > 1\right\} \leq \sum_{k=1}^n \rho\left(Z_k(v)\right),
\end{multline}
where 
\[
    \rho(x) = \begin{cases}
        0 & x\leq 1/2,\\
        2x - 1 & x\in (1/2, 1],\\
        1 & x> 1.
    \end{cases}
\]
Note that $\rho$ is Lipschitz continuous with Lipschitz constant equal to two and that $\mathbf{1}\{x\geq 1/2\}\geq \rho(x) \geq \mathbf{1}\{x\geq 1\}$. In view of Markov's inequality and assumption \eqref{eq:hypercontract}, 
\begin{equation}
\label{rhoexpectation}
    \mb{E} \rho\big(Z_k(v)\big) \leq \mb{P}\left(\langle X, v\rangle^2 > {\frac{\kappa^2}{2}\|\bs{\Sigma}\|}\sqrt{\frac{n}{r(\bs{\Sigma})}}\right)\leq \frac{4\mb{E}\langle X,v\rangle^4}{\kappa ^4 \|\bs{\Sigma}\|^2}\cdot \frac{r(\bs{\Sigma})}{n}\leq \frac{4r(\bs{\Sigma})}{n}.
\end{equation}
We deduce that 
\begin{equation*}
    \mb{P}\left( \sup_{\|v\|_2 = 1}|I_v| > c'n\right)
    \leq \mb{P} \left(\sup_{\|v\|_2=1}\sum_{k=1}^n \rho(Z_k(v))>c'n\right)
    \leq \mb{P} \Bigg(S>c'n - {4r(\bs{\Sigma})}\Bigg).
\end{equation*}
We will now estimate the right-hand side in the display above using Bousquet's version of Talagrand's inequality (Lemma \ref{talagrand-bousquet} in the Appendix). Denote 
\[
\sigma^2 = \sup_{\|v\|_2 = 1}\Var\big(\rho(Z_1(v))\big)
\]
and, observe that in view of inequality \eqref{eq:talagrand:corollary} in the Appendix applied with $U=1$ and $\sigma^2_\ast = n\sigma^2$, $\mb{P} \left(S>c'n - 4r(\bs{\Sigma})\right)\leq e^{-t}$ whenever
\[
     {4r(\bs{\Sigma})} + 2\mb{E}S + \sigma\sqrt{2tn} +  4t/3 \leq c'n.
\]
To prove that this relation holds for suitable choices of parameters $c$ and $t$ (where $r(\bs{\Sigma})/n\leq c$), first note that 
\begin{equation*}
    \Var\big(\rho(Z_1(v))\big) \leq \mb{E}(\rho(Z_1(v))^2\leq \mb{E}\rho(Z_1(v)) \leq \frac{4r(\bs{\Sigma})}{n},
\end{equation*}
where the last inequality follows from the bound \eqref{rhoexpectation}. Therefore, 
\[
\sigma^2 \leq 4r(\bs{\Sigma})/n.
\]
Next, we will estimate $\mb{E}S$. 
Let $\varepsilon_1,\dots \varepsilon_n$ be a sequence of independent random signs. The standard argument based on application of symmetrization and contraction inequalities \citep[Theorem 4.4]{ledoux2013probability}, together with the fact that $\rho(x)$ is Lipschitz continuous with Lipschitz constant equal to two, yields that 
\begin{align*}
    \mb{E}S 
    \leq 2\mb{E}\sup_{\|v\|_2 = 1}\left|\sum_{k=1}^n \varepsilon_k\rho(Z_k)\right|
    &\leq 8\mb{E}\sup_{\|v\|_2 = 1}\left|\sum_{k=1}^n \varepsilon_k \frac{|\langle X_k, v\rangle|}{\left({\kappa^2\|\bs{\Sigma}\|}\sqrt{\frac{n}{r(\bs{\Sigma})}}\right)^{1/2}}\right|
    \\
    &\leq \frac{16}{\left({\kappa^2\|\bs{\Sigma}\|}\sqrt{\frac{n}{r(\bs{\Sigma})}}\right)^{1/2}}\mb{E}\sup_{\|v\|_2 = 1}\left|\sum_{k=1}^n \varepsilon_k \langle X_k, v\rangle\right|.
\end{align*}
Note that 
\begin{multline*}
\mb{E}\sup_{\|v\|_2 = 1}\left|\sum_{k=1}^n \varepsilon_k \langle X_k, v\rangle\right| = \mb E \l\| \sum_{k=1}^n \varepsilon_k  X_k \r\|_2 \leq \mb E^{1/2}\l\| \sum_{k=1}^n  \varepsilon_k  X_k  \r\|_2^2
\\
=\l( \sum_{k=1}^n \mb E \|X_k\|_2^2 \r)^{1/2} = \l( n \|\bs{\Sigma}\| r(\bs{\Sigma})\r)^{1/2},
\end{multline*}
hence we conclude that 
\begin{equation} 
     \mb{E} S \leq \frac{16}{\kappa} n^{1/4}(r(\bs{\Sigma}))^{3/4}.
\end{equation}
As a consequence, we have to choose $c$ and $t$ such that 
\begin{equation}    
\label{eq:t:rank}
4r(\bs{\Sigma})+32n^{1/4}(r(\bs{\Sigma}))^{3/4}+2\sqrt{2t}\sqrt{r(\bs{\Sigma})}+4t/3\leq c'n,
\end{equation}
which is satisfied if both $r(\bs{\Sigma})$ and $t$ do not exceed a constant times $n$ (to be specific, set $t=c_1 n$ for $c_1$ small enough). We conclude that $m\leq c'n$ with probability at least $1-e^{-c_1 n}$ and that inequality \eqref{mbound} holds with probability at least $1-c'(p)/n$ whenever $r(\bs{\Sigma})<cn$. 
Combining this result with the estimates \eqref{fsbound} and \eqref{peakyfbound}, we deduce after a small computation that with probability at least $1-c'(p)/n$, the ``peaky'' term admits the upper bound of the form 
\begin{multline}
\label{peakyfinalbound}
\sup_{\|v\|_2 = 1} \frac{1}{n}\sum_{k=1}^n \langle X_k, v\rangle^2 \mathbf{1}\l\{\langle X_k, v\rangle^2 > \kappa^2\|\bs{\Sigma}\|\sqrt{\frac{n}{r(\bs{\Sigma})}} \r\}
\\ 
\leq \frac{f(m,[n])}{n}\leq C_2(p,\kappa) \left(\frac{\max_{k} \|X_k\|_2^2}{n} + \|\bs{\Sigma}\|\sqrt{\frac{r(\bs{\Sigma})}{n}}\right).
\end{multline}
Combining the estimates \eqref{spreadfinalbound} and \eqref{peakyfinalbound} with the decomposition \eqref{peaky-spreadsplit}, we conclude that with probability at least $1-e^{-r(\bs{\Sigma})} - c'(p)/n$,
\begin{equation}
\label{eq:cov:error}
    \left\|\frac{1}{n}\sum_{k=1}^n X_k X_k^\top - \bs{\Sigma} \right\|\leq 
    C(p,\kappa)\left(\frac{\max_{k} \|X_k\|_2^2}{n} + \|\bs{\Sigma}\|\sqrt{\frac{r(\bs{\Sigma})}{n}}\right).
\end{equation}
To obtain the desired bound in expectation, let $A$ be the event of probability at least $1-e^{-r(\bs{\Sigma})} - c'(p)/n$ on which inequality \eqref{eq:cov:error} holds. Then
\begin{multline*}
    \mb{E}^{1/2}\left\|\frac{1}{n}\sum_{k=1}^n X_k X_k^\top - \bs{\Sigma} \right\|^2 
    \leq C_1(p,\kappa)\left(\frac{\mb E^{1/2}\max_{k} \|X_k\|_2^4}{n} + \|\bs{\Sigma}\|\sqrt{\frac{r(\bs{\Sigma})}{n}}\right) 
    \\
    + \mb{E}^{1/2}\left[\left\|\frac{1}{n}\sum_{k=1}^n X_k X_k^\top - \bs{\Sigma} \right\|^2\mathbf{1}(A^c)\right].
\end{multline*}
H\"older's inequality implies that 
\begin{equation*}
  \mb{E}^{1/2}\left[\left\|\frac{1}{n}\sum_{k=1}^n X_k X_k^\top - \bs{\Sigma} \right\|^2\mathbf{1}(A^c)\right]
  \leq 
  \mb{E}^{2/p}\left\|\frac{1}{n}\sum_{k=1}^n X_k X_k^\top - \bs{\Sigma} \right\|^{p/2}\left(\frac{c'(p)}{n}+ e^{-r(\bs{\Sigma})}\right)^{\frac{p-4}{2p}}.
\end{equation*}
Finally, we invoke Rosenthal's inequality \eqref{ineq:rosenthal1} to deduce that
\begin{multline*}
\mb{E}^{2/p}\left\|\frac{1}{n}\sum_{k=1}^n X_k X_k^\top - \bs{\Sigma} \right\|^{p/2}
\leq C(p,\kappa)\Bigg[ \sqrt{\frac{r(\bs{\Sigma})}{n}}\|\bs{\Sigma}\|\sqrt{\log(er(\bs{\Sigma}))} 
\\
+ \frac{\log(er(\bs{\Sigma}))}{n}\mb{E}\max_{k}\|X_k\|_2^2
+ \frac{1}{n} \mb{E}^{2/p}\max_{k}\|X_k\|_2^p \Bigg].
\end{multline*}
To deduce the bound above, we set $\W_k = \frac{1}{n}X_k X_k^{\top}$ and  $\V_n^2 = \frac{1}{n}\kappa^4 \mathrm{trace}(\bs{\Sigma})\bs{\Sigma}$ in \eqref{ineq:rosenthal1}, which is possible due to inequality \eqref{eq:4th:moment}. This yields the relations $\sigma^2 = \frac{1}{n} \kappa^4 \|\bs{\Sigma}\|^2 r(\bs{\Sigma})$ and $r(\V_n^2) = r(\bs{\Sigma})$.
Since $r(\bs{\Sigma})<cn$ by assumption, $\log(er(\bs{\Sigma}))<C(p) \l(n\vee e^{r(\bs{\Sigma})} \r)^{(p-4)/2p}$, implying that
\begin{multline*}
\mb{E}^{1/2}\left\|\frac{1}{n}\sum_{k=1}^n X_k X_k^\top -\bs{\Sigma} \right\|^2
\leq
\\
C(\kappa,p) \Bigg( \|\bs{\Sigma}\|\sqrt{\frac{r(\bs{\Sigma})}{n}}
+ \frac{\mb{E}^{1/2}\max\limits_{k}\|X_k\|_2^4}{n} \bigvee \frac{\mb{E}^{2/p}\max\limits_{k}\|X_k\|_2^p}{n}\left(\frac{1}{n} + e^{-r(\bs{\Sigma})}\right)^{\frac{p-4}{2p}}\Bigg).
\end{multline*}
From this the claim follows.
\end{proof}

\subsection{Empirical eigenvector estimation}
\label{section:examples:eigenvector}

In this section, we continue the considerations of Section \ref{sec:covariance}. Let $X\in \mb R^d$ be a random vector such that $\mb{E}X=0$ and $\mb{E} XX^\top = \bs{\Sigma}$, and let $\lambda_1\geq\dots\geq \lambda_d$ and $u_1,\dots,u_d$ be the eigenvalues and unit eigenvectors of $\bs{\Sigma}$, respectively. Without loss of generality we assume that $\lambda_d>0$. Moreover, let $g_1 = \lambda_1 - \lambda_2$ and $g_j = \min(\lambda_{j-1} - \lambda_j, \lambda_j - \lambda_{j+1})$ for $j = 2,\dots,d$ be the different spectral gaps, meaning that if $g_j>0$, then the eigenvector $u_j$ is uniquely determined up to the sign. Given a sequence $X_1,\dots, X_n\in \mb{R}^d$ of i.i.d.~copies of $X$, let  $\hat\lambda_1\geq\dots\geq \hat\lambda_d$ and $\hat u_1,\dots,\hat u_d$ be the eigenvalues and eigenvectors of the empirical covariance matrix $\hat{\bs{\Sigma}} = \frac{1}{n}\sum_{k = 1}^n X_k X_k^{\top}$, respectively. A question that has been studied for decades asks for perturbation bounds for the empirical spectral characteristics, we refer for instance to \cite{anderson1963,horvath_kokoszka_book_2012,Hsing:eubank:book:2015} for some classical and more recent results and applications. For the special case of spectral projectors, the Davis-Kahan inequality (cf.~\cite{davis:kahan:1970}) is among the most prominent tools, see \citep{jirak:wahl:proceedings:2020,samworth:davis:kahan:2015} and the references therein for some recent context. Combining the Davis-Kahan inequality with Theorem \ref{thm:covariance:estimation} we get that
\begin{align}\label{eq:davis:kahan:classic}
    \E^{1/2} \left\|\hat u_j\hat u_j^\top - u_ju_j^\top \right\|_2^2\leq C(\kappa,p) \left( \frac{\|\bs{\Sigma}\|}{g_j}\sqrt{\frac{r(\bs{\Sigma})}{n}}
+ \frac{\mb{E}^{2/p}\max_{k=1,\ldots,n}\|X_k\|_2^p}{ng_j}\right),
\end{align}
provided that $g_j>0$. Note that the requirement $\frac{r(\bs{\Sigma})}{n}\leq c$ can be dropped in this case because the left-hand side is always bounded by $\sqrt{2}$. Although prominent, estimates of the type \eqref{eq:davis:kahan:classic} are often sub-optimal, we refer to \citep{jirak:wahl:proceedings:2020,jirak:wahl:advances:2023} for a detailed discussion. It turns out that the complexity of the problem is captured by the \textit{relative rank}
\begin{align}\label{relative:rank}
r_j(\bs{\Sigma}) = \sum_{i \neq j} \frac{\lambda_i}{|\lambda_i - \lambda_j|} + \frac{\lambda_j}{g_j},
\end{align}
in contrast to the effective rank $r(\bs{\Sigma})$. The following result improves upon the bounds in \citep[][Corollary 5]{jirak:wahl:advances:2023} in terms of less moments. 

\begin{theorem}
\label{thm:davis:kahan:proj}
Assume that for some $p>4$ condition \eqref{eq:hypercontract} holds. Let $j\in\{1,\dots,d\}$ be such that $g_j>0$. Then 
\begin{align*}
\E^{1/2} \left\|\hat u_j - u_j\right\|_2^2\leq \E^{1/2}\left\|\hat u_j\hat u_j^\top - u_ju_j^\top\right\|_2^2 &\leq C(\kappa,p) \left(\sqrt{\frac{\lambda_j}{g_j}}\sqrt{\frac{r_j(\bs{\Sigma})}{n}} + \frac{r_j(\bs{\Sigma})}{n^{1-2/p}}\right),
\end{align*}
where $\|\cdot\|_2$ also denotes the Frobenius norm for matrices and the sign of $\hat u_j$ is chosen such that $\langle \hat u_j,u_j \rangle\geq 0$.
\end{theorem}

\begin{remark}\label{remark:KL:coefficients}
    The random vector $X$ can be expressed in terms of the Karhunen-Lo\`{e}ve decomposition $X = \sum_{i=1}^d \sqrt{\lambda_i} \eta_iu_i$, where the Karhunen-Lo\`{e}ve coefficients $\eta_1,\dots,\eta_d$ are uncorrelated with $\E \eta_i^2 = 1$ and defined by $\eta_i=\langle X,u_i\rangle/\sqrt{\lambda_i}$. In this case, \eqref{eq:hypercontract} holds if $\eta_1,\dots,\eta_d$ form a martingale difference sequence with $\max_{i\leq d}\E^{1/p} |\eta_i|^p \leq \kappa/\sqrt{p-1}$, as can be seen by an application of Burkholder's inequality (Theorem 2.1 in \cite{rio:2009:moment}).
\end{remark}

\begin{remark}
   Since $\lambda_j \leq \|\bs{\Sigma}\|$ and $|\lambda_k - \lambda_j|\geq g_j$ for $k \neq j$, we have the inequality
\begin{align*}
\sqrt{\frac{\lambda_j}{g_j}} \sqrt{\frac{r_j(\bs{\Sigma})}{n}} \leq \frac{\|\bs{\Sigma}\|}{g_j} \sqrt{\frac{r(\bs{\Sigma})}{n}},
\end{align*}
meaning that the first term on the right-hand side of Theorem \ref{thm:davis:kahan:proj} is always an improvement over the corresponding one in \eqref{eq:davis:kahan:classic}.
\end{remark}

\begin{proof}[Proof of Theorem \ref{thm:davis:kahan:proj}]
By display (5.24) in \cite{Hsing:eubank:book:2015} combined with Proposition 1 in \cite{JW24} (see also Lemma 2 in \cite{jirak:wahl:advances:2023}), we have
\begin{align}
\label{lem:relative:davis:kahan}
\left\|\hat u_j - u_j\right\|_2 \leq \left\|\hat u_j\hat u_j^\top - u_ju_j^\top\right\|_2 \leq 4\sqrt{2} \left\|\T_j(\hat{\bs{\Sigma}} - \bs{\Sigma})\T_j\right\|
\end{align}
with 
\begin{align*}
\T_j = |\R_j|^{1/2} + g_j^{-1/2}u_ju_j^\top,\qquad |\R_j|^{1/2} = \sum_{i \neq j} \frac{1}{\sqrt{|\lambda_i - \lambda_j|}}u_iu_i^\top.
\end{align*}
Since $\T_j$ is symmetric, 
\begin{align*}
\T_j\left(\hat{\bs{\Sigma}} - \bs{\Sigma}\right)\T_j = \frac{1}{n}\sum_{k=1}^n (\T_j X_k ) (\T_j X_k)^{\top} - \mb{E} (\T_j X ) (\T_j X)^{\top}.
\end{align*}
Thus, estimating $\|\T_j(\hat{\bs{\Sigma}} - \bs{\Sigma})\T_j\|$ is again a covariance estimation problem with a random vector given by 
\begin{align*}
\T_j X = \sum_{i \neq j} \left(\frac{\lambda_i}{|\lambda_i - \lambda_j|}\right)^{1/2} \eta_i u_i + \left(\frac{\lambda_j}{g_j}\right)^{1/2} \eta_j u_j
\end{align*}
with Karhunen-Lo\`{e}ve coefficients $\eta_1,\dots,\eta_d$ introduced in Remark \ref{remark:KL:coefficients}.
We now explore the fact that assumption \eqref{eq:hypercontract} is invariant under linear transformations. Indeed, for any $u \neq 0$, we have
\begin{align*}
\frac{\E^{1/p} |\langle \T_j X, u \rangle|^p}{\E^{1/2}  \langle \T_j X, u \rangle^2} =  \frac{\E^{1/p} |\langle X, \T_j u \rangle|^p}{\E^{1/2}  \langle X, \T_j u \rangle^2} \leq \sup_{\|v\|_2 = 1} \frac{\E^{1/p} |\langle X, v \rangle|^p}{\E^{1/2} \langle X, v \rangle^2}\leq \kappa,
\end{align*}
implying that \eqref{eq:hypercontract} also holds for $\T_j X$. In addition, setting $v = u_j$ in \eqref{eq:hypercontract}, we get
\begin{align*}
\max_{i=1,\ldots,d}\E^{1/p}|\eta_i|^p \leq \kappa. 
\end{align*}
Using the triangle inequality, it follows that
\begin{align*}
\E^{2/p} \left\|\T_j X\right\|^p &= \E^{2/p}\left(\sum_{i \neq j} \frac{\lambda_i}{|\lambda_i - \lambda_j|} \eta_i^2  + \frac{\lambda_j}{g_j}\eta_j^2\right)^{p/2} \\&\leq \sum_{i \neq j} \frac{\lambda_i}{|\lambda_i - \lambda_j|} \E^{2/p} |\eta_i|^p + \frac{\lambda_j}{g_j}\E^{2/p}|\eta_j|^p \leq r_j(\bs{\Sigma}) \kappa^{2}.
\end{align*}
This in turn yields the estimate
\begin{align}\label{eq:thm:davis:kahan:proj:2}
\E^{2/p} \max_{i=1,\ldots,n}\left\|\T_j X_i\right\|^p \leq n^{2/p} r_j(\bs{\Sigma}) \kappa^{2}.
\end{align}
Moreover, we have the relations
\begin{align}\label{eq:thm:davis:kahan:proj:4}
\operatorname{trace}\left(\T_j \bs{\Sigma} \T_j\right) = r_j(\bs{\Sigma}),\qquad \left\|\T_j \bs{\Sigma} \T_j\right\| = \left(\max_{i \neq j} \frac{\lambda_i}{|\lambda_i - \lambda_j|}\right) \bigvee \frac{\lambda_j}{g_j} \leq \frac{2\lambda_j}{g_j}.
\end{align}
Theorem \ref{thm:covariance:estimation}, together with displays \eqref{eq:thm:davis:kahan:proj:2} and \eqref{eq:thm:davis:kahan:proj:4}, now yields the inequality
\begin{align*}
\E^{1/2} \left\|\hat u_j\hat u_j^\top - u_ju_j^\top\right\|_2^2&\leq  
 4\sqrt{2}\E^{1/2} \|\T_j(\hat{\bs{\Sigma}} - \bs{\Sigma})\T_j\|^2 \\ &\leq C(\kappa,p) \left(\|\T_j \bs{\Sigma} \T_j\| \sqrt{\frac{r(\T_j \bs{\Sigma} \T_j)}{n}} + \frac{\E^{2/p} \max_{i} \left\|\T_j X_i\right\|_2^p}{n}\right)\\
&\leq C'(\kappa,p) \left(\sqrt{\frac{\lambda_j}{g_j}}\sqrt{\frac{r_j(\bs{\Sigma})}{n}} + \frac{r_j(\bs{\Sigma})}{n^{1-2/p}}\right), 
\end{align*}
provided that $\frac{r_j(\bs{\Sigma})}{n}\leq c$. Finally, in the case where $\frac{r_j(\bs{\Sigma})}{n}>c$, we trivially have
\begin{align*}
    \E^{1/2} \left\|\hat u_j\hat u_j^\top - u_ju_j^\top\right\|_2^2\leq \sqrt{2}\leq \frac{\sqrt{2}}{c}\frac{r_j(\bs{\Sigma})}{n^{1-2/p}}.
\end{align*}
The conclusion follows.
\end{proof}


\bibliographystyle{apalike}
\bibliography{ref}

\begin{thebibliography}{}

\bibitem[Abdalla and Zhivotovskiy, 2024]{abdalla2022covariance}
Abdalla, P. and Zhivotovskiy, N. (2024).
\newblock Covariance estimation: Optimal dimension-free guarantees for
  adversarial corruption and heavy tails.
\newblock {\em Journal of the European Mathematical Society}.

\bibitem[Adamczak, 2008]{adamczak2008tail}
Adamczak, R. (2008).
\newblock A tail inequality for suprema of unbounded empirical processes with
  applications to {M}arkov chains.
\newblock {\em Electronic Journal of Probability}, 13:1000--1034.

\bibitem[Adamczak, 2010]{adamczak2010few}
Adamczak, R. (2010).
\newblock A few remarks on the operator norm of random {T}oeplitz matrices.
\newblock {\em Journal of Theoretical Probability}, 23(1):85--108.

\bibitem[Adamczak et~al., 2010]{adamczak2010quantitative}
Adamczak, R., Litvak, A., Pajor, A., and Tomczak-Jaegermann, N. (2010).
\newblock Quantitative estimates of the convergence of the empirical covariance
  matrix in log-concave ensembles.
\newblock {\em Journal of the American Mathematical Society}, 23(2):535--561.

\bibitem[Adamczak et~al., 2011]{adamczak2011sharp}
Adamczak, R., Litvak, A.~E., Pajor, A., and Tomczak-Jaegermann, N. (2011).
\newblock Sharp bounds on the rate of convergence of the empirical covariance
  matrix.
\newblock {\em Comptes Rendus. Math{\'e}matique}, 349(3-4):195--200.

\bibitem[Ahlswede and Winter, 2002]{ahlswede2002strong}
Ahlswede, R. and Winter, A. (2002).
\newblock Strong converse for identification via quantum channels.
\newblock {\em IEEE Transactions on Information Theory}, 48(3):569--579.

\bibitem[Anderson, 1963]{anderson1963}
Anderson, T.~W. (1963).
\newblock Asymptotic theory for principal component analysis.
\newblock {\em Ann. Math. Statist.}, 34:122--148.

\bibitem[Bai and Yin, 2008]{bai2008limit}
Bai, Z.-D. and Yin, Y.-Q. (2008).
\newblock Limit of the smallest eigenvalue of a large dimensional sample
  covariance matrix.
\newblock In {\em Advances In Statistics}, pages 108--127.

\bibitem[Bakhshizadeh et~al., 2023]{bakhshizadeh2023sharp}
Bakhshizadeh, M., Maleki, A., and De~La~Pena, V.~H. (2023).
\newblock Sharp concentration results for heavy-tailed distributions.
\newblock {\em Information and Inference: A Journal of the IMA},
  12(3):1655--1685.

\bibitem[Boucheron et~al., 2005]{boucheron2005moment}
Boucheron, S., Bousquet, O., Lugosi, G., and Massart, P. (2005).
\newblock Moment inequalities for functions of independent random variables.
\newblock {\em The Annals of Probability}, 33(2):514--560.

\bibitem[Bourgain, 1996]{bourgain1996random}
Bourgain, J. (1996).
\newblock Random points in isotropic convex sets.
\newblock {\em Convex geometric analysis}, 34:53--58.

\bibitem[Bousquet, 2003]{bousquet2003concentration}
Bousquet, O. (2003).
\newblock Concentration inequalities for sub-additive functions using the
  entropy method.
\newblock In {\em Stochastic inequalities and applications}, pages 213--247.
  Springer.

\bibitem[Chen et~al., 2012]{chen2012masked}
Chen, R.~Y., Gittens, A., and Tropp, J.~A. (2012).
\newblock The masked sample covariance estimator: an analysis using matrix
  concentration inequalities.
\newblock {\em Information and Inference: A Journal of the IMA}, 1(1):2--20.

\bibitem[Davis and Kahan, 1970]{davis:kahan:1970}
Davis, C. and Kahan, W.~M. (1970).
\newblock The rotation of eigenvectors by a perturbation. {III}.
\newblock {\em SIAM Journal of Numerical Analysis}, 7:1--46.

\bibitem[Dirksen, 2011]{dirksen2011noncommutative}
Dirksen, S. (2011).
\newblock Noncommutative and vector-valued {R}osenthal inequalities.
\newblock {\em Ph.D. thesis}.

\bibitem[Einmahl and Li, 2008]{einmahl2008characterization}
Einmahl, U. and Li, D. (2008).
\newblock Characterization of {LIL} behavior in {B}anach space.
\newblock {\em Transactions of the American Mathematical Society},
  360(12):6677--6693.

\bibitem[Figiel et~al., 1997]{figiel1997extremal}
Figiel, T., Hitczenko, P., Johnson, W., Schechtman, G., and Zinn, J. (1997).
\newblock Extremal properties of {R}ademacher functions with applications to
  the {K}hintchine and {R}osenthal inequalities.
\newblock {\em Transactions of the American Mathematical Society},
  349(3):997--1027.

\bibitem[Fuk and Nagaev, 1971]{fuk1971probability}
Fuk, D.~K. and Nagaev, S.~V. (1971).
\newblock Probability inequalities for sums of independent random variables.
\newblock {\em Theory of Probability \& Its Applications}, 16(4):643--660.

\bibitem[Gin{\'e} et~al., 2000]{gine2000exponential}
Gin{\'e}, E., Latala, R., and Zinn, J. (2000).
\newblock Exponential and moment inequalities for {U}-statistics.
\newblock In {\em High Dimensional Probability II}, pages 13--38. Springer, New
  York.

\bibitem[Gu{\'e}don et~al., 2017]{guedon2017interval}
Gu{\'e}don, O., Litvak, A.~E., Pajor, A., and Tomczak-Jaegermann, N. (2017).
\newblock On the interval of fluctuation of the singular values of random
  matrices.
\newblock {\em Journal of the European Mathematical Society}, 19(5).

\bibitem[Han, 2022]{han2022exact}
Han, Q. (2022).
\newblock Exact spectral norm error of sample covariance.
\newblock {\em arXiv preprint arXiv:2207.13594}.

\bibitem[Horv\'{a}th and Kokoszka, 2012]{horvath_kokoszka_book_2012}
Horv\'{a}th, L. and Kokoszka, P. (2012).
\newblock {\em Inference for {f}unctional {d}ata {w}ith {a}pplications}.
\newblock Springer Series in Statistics. Springer, New York.

\bibitem[Hsing and Eubank, 2015]{Hsing:eubank:book:2015}
Hsing, T. and Eubank, R. (2015).
\newblock {\em Theoretical {f}oundations of {f}unctional {d}ata {a}nalysis,
  with an {i}ntroduction to {l}inear {o}perators}.
\newblock Wiley Series in Probability and Statistics. John Wiley \& Sons, Ltd.,
  Chichester, UK.

\bibitem[Jirak and Wahl, 2020]{jirak:wahl:proceedings:2020}
Jirak, M. and Wahl, M. (2020).
\newblock Perturbation bounds for eigenspaces under a relative gap condition.
\newblock {\em Proceedings of the American Mathematical Society},
  148(2):479--494.

\bibitem[Jirak and Wahl, 2023]{jirak:wahl:advances:2023}
Jirak, M. and Wahl, M. (2023).
\newblock Relative perturbation bounds with applications to empirical
  covariance operators.
\newblock {\em Adv. Math.}, 412:Paper No. 108808, 59.

\bibitem[Jirak and Wahl, 2024]{JW24}
Jirak, M. and Wahl, M. (2024).
\newblock Quantitative limit theorems and bootstrap approximations for
  empirical spectral projectors.
\newblock {\em Probability Theory and Related Fields}.

\bibitem[Johnson et~al., 1985]{johnson1985best}
Johnson, W.~B., Schechtman, G., and Zinn, J. (1985).
\newblock Best constants in moment inequalities for linear combinations of
  independent and exchangeable random variables.
\newblock {\em The Annals of Probability}, pages 234--253.

\bibitem[Junge and Zeng, 2013]{junge2013noncommutative}
Junge, M. and Zeng, Q. (2013).
\newblock Noncommutative {B}ennett and {R}osenthal inequalities.
\newblock {\em The Annals of Probability}, 41(6):4287--4316.

\bibitem[Kannan et~al., 1997]{kannan1997random}
Kannan, R., Lov{\'a}sz, L., and Simonovits, M. (1997).
\newblock Random walks and an ${O}^*(n^5)$ volume algorithm for convex bodies.
\newblock {\em Random Structures \& Algorithms}, 11(1):1--50.

\bibitem[Klochkov and Zhivotovskiy, 2020]{klochkov2020uniform}
Klochkov, Y. and Zhivotovskiy, N. (2020).
\newblock Uniform {H}anson-{W}right type concentration inequalities for
  unbounded entries via the entropy method.
\newblock {\em Electronic Journal of Probability}, 25(22):1--30.

\bibitem[Koltchinskii, 2011]{koltchinskii2011neumann}
Koltchinskii, V. (2011).
\newblock Von {N}eumann entropy penalization and low-rank matrix estimation.
\newblock {\em The Annals of Statistics}, 39(6):2936--2973.

\bibitem[Koltchinskii and Lounici, 2017]{koltchinskii2017concentration}
Koltchinskii, V. and Lounici, K. (2017).
\newblock Concentration inequalities and moment bounds for sample covariance
  operators.
\newblock {\em Bernoulli}, 23(1):110--133.

\bibitem[Ledoux and Talagrand, 1991]{ledoux2013probability}
Ledoux, M. and Talagrand, M. (1991).
\newblock {\em Probability in Banach {s}paces: {i}soperimetry and {p}rocesses}.
\newblock Springer Science \& Business Media, New York.

\bibitem[Mendelson and Paouris, 2014]{mendelson2014singular}
Mendelson, S. and Paouris, G. (2014).
\newblock On the singular values of random matrices.
\newblock {\em Journal of the European Mathematical Society}, 16(4).

\bibitem[Minsker, 2017]{minsker2011some}
Minsker, S. (2017).
\newblock On some extensions of {B}ernstein's inequality for self-adjoint
  operators.
\newblock {\em Statistics \& Probability Letters}, 127:111--119.

\bibitem[Nagaev, 1979]{nagaev1979large}
Nagaev, S.~V. (1979).
\newblock Large deviations of sums of independent random variables.
\newblock {\em The Annals of Probability}, pages 745--789.

\bibitem[Oliveira, 2010]{oliveira2010sums}
Oliveira, R.~I. (2010).
\newblock Sums of random {H}ermitian matrices and an inequality by {R}udelson.
\newblock {\em Electronic Communications in Probability}, 15:203--212.

\bibitem[Pinelis and Utev, 1985]{pinelis1985estimates}
Pinelis, I. and Utev, S. (1985).
\newblock Estimates of the moments of sums of independent random variables.
\newblock {\em Theory of Probability \& Its Applications}, 29(3):574--577.

\bibitem[Rio, 2009]{rio:2009:moment}
Rio, E. (2009).
\newblock Moment inequalities for sums of dependent random variables under
  projective conditions.
\newblock {\em Journal of Theoretical Probability}, 22(1):146--163.

\bibitem[Rio, 2017]{rio2017constants}
Rio, E. (2017).
\newblock About the constants in the {F}uk-{N}agaev inequalities.
\newblock {\em Electronic Communications in Probability}, 22(28):12p.

\bibitem[Rudelson, 1999]{rudelson1999random}
Rudelson, M. (1999).
\newblock Random vectors in the isotropic position.
\newblock {\em Journal of Functional Analysis}, 164(1):60--72.

\bibitem[Rudelson and Vershynin, 2007]{rudelson2007sampling}
Rudelson, M. and Vershynin, R. (2007).
\newblock Sampling from large matrices: an approach through geometric
  functional analysis.
\newblock {\em Journal of the ACM}, 54(4).

\bibitem[Sharakhmetov, 1998]{ibragimov1998exact}
Sharakhmetov, S. (1998).
\newblock On an exact constant for the {R}osenthal inequality.
\newblock {\em Theory of Probability \& Its Applications}, 42(2):294--302.

\bibitem[Srivastava and Vershynin, 2013]{srivastava2013covariance}
Srivastava, N. and Vershynin, R. (2013).
\newblock Covariance estimation for distributions with $2+\varepsilon$ moments.
\newblock {\em The Annals of Probability}, 41(5):3081--3111.

\bibitem[Talagrand, 1996]{talagrand1996new}
Talagrand, M. (1996).
\newblock New concentration inequalities in product spaces.
\newblock {\em Inventiones mathematicae}, 126(3):505--563.

\bibitem[Talagrand, 2014]{talagrand2014upper}
Talagrand, M. (2014).
\newblock {\em {U}pper and {l}ower {b}ounds for {s}tochastic {p}rocesses},
  volume~60.
\newblock Springer, New York.

\bibitem[Tikhomirov, 2018]{tikhomirov2018sample}
Tikhomirov, K. (2018).
\newblock Sample covariance matrices of heavy-tailed distributions.
\newblock {\em International Mathematics Research Notices},
  2018(20):6254--6289.

\bibitem[Tropp, 2008]{tropp2008norms}
Tropp, J.~A. (2008).
\newblock Norms of random submatrices and sparse approximation.
\newblock {\em Comptes Rendus. Math{\'e}matique}, 346(23-24):1271--1274.

\bibitem[Tropp, 2012]{tropp2012user}
Tropp, J.~A. (2012).
\newblock User-friendly tail bounds for sums of random matrices.
\newblock {\em Foundations of computational mathematics}, 12:389--434.

\bibitem[Tropp, 2015]{tropp2015introduction}
Tropp, J.~A. (2015).
\newblock An introduction to matrix concentration inequalities.
\newblock {\em Foundations and Trends{\textregistered} in Machine Learning},
  8(1-2):1--230.

\bibitem[Utev, 1985]{utev1985extremal}
Utev, S. (1985).
\newblock Extremal problems in moment inequalities.
\newblock {\em Trudy Inst. Mat. Sib. Otd. AN SSSR}, 5:56--75.

\bibitem[van~der Vaart and Wellner, 1996]{Vaart1996Weak-convergenc}
van~der Vaart, A.~W. and Wellner, J.~A. (1996).
\newblock {\em Weak {c}onvergence and {e}mpirical {p}rocesses}.
\newblock Springer Series in Statistics. Springer-Verlag, New York.

\bibitem[Vershynin, 2011]{vershynin2011approximating}
Vershynin, R. (2011).
\newblock Approximating the moments of marginals of high-dimensional
  distributions.
\newblock {\em The Annals of Probability}, 39(4):1591--1606.

\bibitem[Vershynin, 2012]{vershynin2012close}
Vershynin, R. (2012).
\newblock How close is the sample covariance matrix to the actual covariance
  matrix?
\newblock {\em Journal of Theoretical Probability}, 25(3):655--686.

\bibitem[Wainwright, 2019]{wainwright2019high}
Wainwright, M.~J. (2019).
\newblock {\em High-dimensional {s}tatistics: a {n}on-asymptotic {v}iewpoint},
  volume~48.
\newblock Cambridge University Press, Cambridge, UK.

\bibitem[Wei and Minsker, 2017]{wei2017estimation}
Wei, X. and Minsker, S. (2017).
\newblock Estimation of the covariance structure of heavy-tailed distributions.
\newblock {\em Advances in neural information processing systems}, 30.

\bibitem[Yu et~al., 2015]{samworth:davis:kahan:2015}
Yu, Y., Wang, T., and Samworth, R.~J. (2015).
\newblock A useful variant of the {D}avis-{K}ahan theorem for statisticians.
\newblock {\em Biometrika}, 102(2):315--323.

\end{thebibliography}

\appendix

\section{Auxiliary results}
\label{sec:tech}

In the Appendix, we compile the background material and technical results on which our arguments rely.
		
\begin{lemma} 
(\cite[Theorem 3.1]{minsker2011some} and \cite[Theorem 7.7.1]{tropp2015introduction})
\label{matrixBernsteinErank}
Let $\W_1, \dots, \W_n\in \mb{C}^{d\times d}$ be a sequence of independent, centered, self-adjoint random matrices such that $\|\W_k\|\leq U, \ k=1,\ldots,n$ almost surely. Assume that $\V_n^2\succeq \sum_k \mb{E}\W_k^2$ and let $\sigma^2 = \|\V_n^2\|$. Then for any $t\geq \sigma+U/3$,
\begin{equation}
\label{eq:matrix:Bernstein}
    \mb{P}\left(\left\|\sum_{k=1}^n \W_k\right\|>t\right) \leq 4
    r\left(\V_n^2\right) \exp \left[-\frac{t^2 / 2}{\sigma^2+t U / 3}\right].
\end{equation}
\end{lemma}

\begin{corollary}
\label{cor:matrixBernsteinLp}
Let $p\geq 1$. Under the assumptions of \Cref{matrixBernsteinErank}, 
\begin{equation*}
\mb{E}^{1/p}\left\|\sum_{k=1}^n \W_k\right\|^p
\leq K\l( \sigma \sqrt{q} + Uq \r)
\end{equation*}
where $q = \log(er(\V_n^2))\vee p$ and $K>0$ is a numerical constant.
\end{corollary}
\begin{proof}
The proof is based on substituting the tail bound \eqref{eq:matrix:Bernstein} into the relation
\[
\mb{E}^{1/p}\left\|\sum_{k=1}^n \W_k\right\|^p = \l(\int_{0}^\infty p t^{p-1}\mb{P}\l(\l\|\sum_{k=1}^n \W_k\r\|> t \r)dt\r)^{1/p}
\]
and following a standard sequence of estimates. Inequalities of this type are well known: see, for instance, \citep[exercise 2.8]{wainwright2019high} or \citep[Lemma 2.3.2]{talagrand2014upper} where the case $p=1$ is considered.
\end{proof}

\begin{lemma}\citep[Hoffmann-J{\o}rgensen inequality:][Proposition 6.7]{ledoux2013probability}
\label{prop:HJtail}
Let $X_1,\ldots,X_n$ be independent, symmetrically distributed random variables with values in a separable Banach space with norm $\|\cdot\|_B$. Set $S_n=\sum_{k=1}^n X_k$. Then for any $s,t>0$, 
\begin{equation*}
    \mathbb{P}\left(\left\|S_n\right\|_B>2 t+s\right) \leq4\left(\mathbb{P}\left(\left\|S_n\right\|_B>t\right)\right)^2+\mathbb{P}\left(\max _{k}\left\|X_k\right\|_B>s\right).
\end{equation*}
\end{lemma}
\begin{lemma}\cite[Proposition 6.8]{ledoux2013probability}
\label{HJexpectation}
Assume that $q>0$, and let $X_1,\ldots,X_n$ be independent, symmetrically distributed random variables with values in a separable Banach space with norm $\|\cdot\|_B$. Set $S_n=\sum_{k=1}^n X_k$. Then for $t_0 = \inf\{t>0: \ \mb{P}(\|S_n\|_B>t)\leq (8\cdot 3^q)^{-1}\}$,
\begin{equation}
    \mathbb{E}\left\|S_n\right\|_B^q \leq 2 \cdot 3^q \mathbb{E} \max_{k}\left\|X_k\right\|_B^q + 2\left(3 t_0\right)^q.
\end{equation}
\end{lemma}

\begin{lemma}\cite[Theorems 6.20 and 6.21]{ledoux2013probability}
\label{HJLp}
Let $X_1,\ldots,X_n$ be independent random variables with values in a separable Banach space with norm $\|\cdot\|_B$. There exists a numerical constant $K$ such that for all $p > 1$,
\begin{align*}
   \mb E^{1/p} \left\|\sum_{k=1}^n X_k\right\|_B^p &\leq K \frac{p}{\log(ep)} \left(\mb E\left\|\sum_{k=1}^n X_k\right\|_B + \mb E^{1/p}\max_{k}\|X_k\|_B^p\right) \text{ and}
   \\
   \l\| \Big\|\sum_{k=1}^n X_k\Big\|_B \r\|_{\psi_1}& \leq K
   \left(\mb E\left\|\sum_{k=1}^n X_k\right\|_B + \l\|\max_{k}\|X_k\|_B \r\|_{\psi_1}\right).
\end{align*}
\end{lemma}

\begin{lemma} 
\label{talagrand-bousquet} 
\cite[Theorem 7.3]{bousquet2003concentration} 
Let $\mc{F}$ be a countable set of measurable real-valued functions and let $X_1,\ldots,X_n$ be i.i.d. Assume that $\mb{E}f(X_1)=0$ for all $f\in \m F$ and that $\sup_{f\in \m F} |f(X_1)|\leq U$ with probability $1$. Denote $Z=\sup _{f \in \mathcal{F}} \sum_{k=1}^n f\left(X_k\right)$.
Assume that $\sigma_\ast^2\geq n\sup_{f\in \mc{F}} \mb{E}f^2(X_1)$ and set $v = \sigma_\ast^2 + 2U\mb{E}[Z]$. Then for all $t\geq 0$, 
\begin{equation}
\label{eq:bousquet}
    \mb{P}\left(Z \geq \mb{E} Z + \sqrt{2 t v}+\frac{tU}{3}\right) \leq e^{-t}.
\end{equation}
\end{lemma}
\noindent Inequality \eqref{eq:bousquet} immediately implies that with probability at least $1-e^{-t}$, 
\begin{equation}
    \label{eq:talagrand:corollary}
Z\leq 2\mb EZ + \sigma_\ast\sqrt{2t} + 4tU/3.
\end{equation}
Indeed, since $\sqrt{a+b}\leq \sqrt{a}+\sqrt{b}$ and $\sqrt{4ab}\leq a+b$ for all $a,b\geq 0$, we can estimate $\sqrt{2tv}$ as 
\[
\sqrt{2tv} \leq \sigma_\ast\sqrt{2t} + \sqrt{4tU \,\mb EZ} \leq \sigma_\ast\sqrt{2t} + \mb EZ + tU, 
\]
which yields \eqref{eq:talagrand:corollary}.
\begin{lemma}
\label{lemma:max-expectation}
(\cite[Lemma 5.1]{rudelson2007sampling} or \cite[Proposition 2.3]{tropp2008norms})
Let $\delta_1,\ldots,\delta_d$ be i.i.d.~Bernoulli random variables with $\mb E\delta_1 = \delta$  such that $\delta \geq 1/d$, and assume that $a_1,\ldots,a_d$ are nonnegative real numbers. Then
\[
\mb E\max_{k=1,\ldots,d} \delta_k a_k \leq \frac{2}{\lfloor \delta^{-1}\rfloor}\sum_{k=1}^{\lfloor \delta^{-1}\rfloor} a_{(k)}
\]
where $a_{(1)}\geq \ldots \geq a_{(d)}$ is a non-increasing rearrangement of $\{a_1,\ldots,a_d\}$.
\end{lemma}
\end{document}